\documentclass[a4paper,11pt]{amsart}

\usepackage{hyperref}
\hypersetup{colorlinks, citecolor=blue, filecolor=black, linkcolor=blue, urlcolor=blue}

\usepackage[top=2.54cm, bottom=2.54cm, outer=2.54cm, inner=2.54cm, headsep=14pt]{geometry} 

\usepackage{hyperref}
\hypersetup{colorlinks, citecolor=blue, filecolor=black, linkcolor=blue, urlcolor=blue}

\usepackage{mathpazo}

\usepackage{enumerate}
\usepackage[shortlabels]{enumitem}
\usepackage{multicol}

\usepackage[latin1]{inputenc}

\usepackage{amsthm} 
\usepackage{amsmath}
\usepackage{amssymb}

\theoremstyle{theorem}
\newtheorem{theorem}{Theorem}[section]
\theoremstyle{definition}
\newtheorem{lemma}[theorem]{Lemma}
\newtheorem{corollary}[theorem]{Corollary}
\newtheorem{proposition}[theorem]{Proposition}
\newtheorem{remark}[theorem]{Remark}
\newtheorem{example}[theorem]{Example}

\newtheorem{definition}[theorem]{Definition}

\usepackage[dvips]{graphicx}

\usepackage{float}

\usepackage[usenames, dvipsnames]{color}

\newcommand{\N}{\mathbb{N}}

\newcommand{\R}{\mathbb{R}}
\newcommand{\B}{\mathbb{B}}
\newcommand{\Rinf}{\mathbb{R}\cup \{+\infty \}}

\newcommand{\I}{\mathcal{I}}

\newcommand{\NN}{\mathcal{N}}
\newcommand{\X}{X}

\newcommand{\eps}{\varepsilon}

\newcommand{\Id}{\mathrm{Id}}

\newcommand{\argmin}{\mbox{\rm argmin\,}}

\newcommand{\bd}{\mbox{\rm bd\,}}
\newcommand{\dom}{\mathop\mathrm{\rm dom}}
\newcommand{\co}{\mbox{\rm co}}
\newcommand{\dist}{\mbox{\rm dist\,}}

\newcommand{\inte}{\mbox{\rm int\,}}

\newcommand{\qri}{\mbox{\rm qri\,}}

\newcommand{\ri}{\mbox{\rm ri\,}}
\newcommand{\rint}{\mbox{\rm ri\,}}

\DeclareMathOperator{\sri}{sri}

\newcommand{\prox}{\mbox{\rm prox}}
\newcommand{\proj}{\mbox{\rm proj}}
\newcommand{\sgn}{\mbox{\rm sgn}}
\newcommand{\soft}{\mbox{\rm soft}}
\newcommand{\supp}{\mbox{\rm supp}}
\newcommand{\esupp}{\mbox{\rm esupp}}

\DeclareFontFamily{U}{matha}{\hyphenchar\font45}
\DeclareFontShape{U}{matha}{m}{n}{<5> <6> <7> <8> <9> <10> gen * matha <10.95> matha10 <12> <14.4> <17.28> <20.74> <24.88> matha12}{}
\DeclareSymbolFont{matha}{U}{matha}{m}{n}

\DeclareMathSymbol{\leq}         {3}{matha}{"A4}
\DeclareMathSymbol{\geq}         {3}{matha}{"A5}

\newcommand{\kin}{{k\in\N}}
\newcommand{\kinn}{{k\in\mathcal{N}}}
\newcommand{\nin}{{n\in\N}}

\usepackage[foot]{amsaddr} 
\usepackage[misc]{ifsym} 

\title{
Thresholding gradient methods in  Hilbert spaces: \\ support identification and linear convergence
}

\author{Guillaume Garrigos}
\address{
\hspace*{-0.5cm} G. Garrigos ~\ \Letter 
\newline \hspace*{0.35cm}
\textup{CNRS, \'Ecole Normale Sup\'erieure (DMA), 
 75005 Paris, France}
\newline \hspace*{0.35cm}
\textup{\nolinkurl{guillaume.garrigos@ens.fr}}
}

\author{Lorenzo Rosasco}
\address{
\hspace*{-0.5cm} L. Rosasco
\newline \hspace*{0.35cm}
\textup{LCSL, Istituto Italiano di Tecnologia and Massachusetts Institute of Technology, Cambridge, MA 02139, USA}
\newline \hspace*{0.35cm}
\textup{Universit\`a degli Studi di Genova (DIBRIS), 16146 Genova, Italy}
\newline \hspace*{0.35cm}
\textup{\nolinkurl{lrosasco@mit.edu}}
}

\author{Silvia Villa}
\address{
\hspace*{-0.5cm} S. Villa
\newline \hspace*{0.35cm}
\textup{Politecnico di Milano (Dipartimento di Matematica), 20133 Milano, Italy}
\newline \hspace*{0.35cm}
\textup{\nolinkurl{silvia.villa@polimi.it}}
}

\begin{document}

\begin{abstract}
We study $\ell^1$ regularized least squares optimization problem in a separable Hilbert space. We show that the iterative soft-thresholding algorithm (ISTA) converges linearly, 
without making any assumption on the linear operator into play or on the problem.
The result is obtained combining two key concepts:  the notion of {\it extended support}, a finite set containing the support, and the notion of  {\it conditioning over finite dimensional sets}.
We prove that ISTA identifies the solution extended support after a finite number of  iterations, and we derive linear convergence from the conditioning property, 
which is always satisfied for $\ell^1$ regularized least squares problems. 
Our analysis extends to the the entire class of thresholding gradient algorithms, for which we provide a conceptually new proof of strong convergence, as well as  convergence rates.

\bigskip

\noindent \textsc{Keywords.} Forward-Backward method, support identification,  conditioning, convergence rates.

\medskip

\noindent \textsc{MSC.}  49K40, 49M29, 65J10, 65J15, 65J20, 65J22, 65K15, 90C25, 90C46.
%
%
%
%
\end{abstract}

\maketitle

\section{Introduction}

Recent works show that, for many problems of interest,   favorable geometry  can greatly improve theoretical results with respect to more general, worst-case perspective \cite{AttBol09,DruLew16,BlaBol16,GarRosVil17}. 
 In this paper,  we follow this perspective to analyze  the convergence properties of threshold gradient methods in separable Hilbert spaces.  
Our starting point is the now classic iterative soft thresholding algorithm (ISTA) to solve the problem
\begin{equation}\label{E:LASSO}
f(x) =  \Vert x \Vert_1 + \frac{1}{2}\Vert Ax-y \Vert^2,
\end{equation}
defined by an operator $A$  on $\ell^2(\N)$ and  where $\Vert \cdot \Vert_1$ is  the $\ell^1$ norm.

From the seminal work  \cite{DauDefDem04}, it is known that ISTA converges strongly in $\ell^2(\N)$. 
This result is generalized  in \cite{ComPes07} to a wider  class of algorithms, the so-called thresholding gradient methods, {noting that these  are special instances of the Forward-Backward algorithm, where the proximal step reduces to a thresholding step onto an orthonormal basis} (Section \ref{S:threshold gradient method}).
Typically,  strong convergence in  Hilbert spaces is the consequence of a particular structure of the considered  problem.  Classic examples being even functions, functions for which the set of minimizers has a nonempty interior, or strongly convex functions \cite{PeySor10}. Further examples are  uniformly convex functions, or functions presenting a favorable geometry around their minimizers, like conditioned functions or Lojasiewicz functions, see e.g. \cite{BauComV2,GarRosVil17}.  Whether the properties of   ISTA,  and more generally threshold gradient methods,  can be explained from this perspective is not apparent from the analysis in 
 \cite{DauDefDem04,ComPes07}. 
 
 Our first contribution is revisiting these results providing 
such an explanation: for these algorithms, the whole sequence of iterates is fully contained in a specific finite-dimensional subspace, ensuring automatically strong convergence. The key argument in our analysis is that  after a finite number of iterations, the iterates identify the so called \textit{extended support} of their limit.
{This set coincides with  the active constraints at the solution of the dual problem, and reduces to the support, if some qualification condition is satisfied.}\\
Going further, we  tackle the question of convergence rates, providing a unifying treatment of finite and infinite dimensional settings. 
In finite dimensions, it is clear that if $A$ is injective, then $f$ becomes a strongly convex function, which guarantees a linear convergence rate. In \cite{HalYinZha07}, it is shown, still in a finite dimensional setting, that the linear rates hold  just assuming $A$ to be injective on the extended support of the problem.
This result is generalized in \cite{BreLor08b} to a Hilbert space setting,  assuming  $A$ to be injective on any subspace of finite support. 
Linear convergence is also obtained by assuming the limit solution to satisfy some nondegeneracy condition \cite{BreLor08b,LiaFadPey14}.
In fact, it was shown recently in \cite{BolNguPeySut16} that, in finite dimension, no assumption at all is needed to guarantee linear rates. Using  a key result in \cite{Li13},   the function $f$ was shown to be  $2$-conditioned on its sublevel sets,  and $2$-conditioning is  sufficient  for  linear rates  \cite{AttBolSva13}. 
Our  identification result, mentioned above, allows to easily bridge the gap between the finite and infinite dimensional settings. Indeed, we show  that in any separable Hilbert space, linear rates of convergence always hold for the soft-thresholding gradient algorithm under no further assumptions. Once again, the key argument to obtain  linear rates is the fact that the iterates generated by the algorithm identify, in finite time, a set on which we know the function to have a favorable geometry.

The paper is organized as follows.
In Section \ref{S:threshold gradient method} we describe  our setting and introduce the thresholding gradient method.
We introduce the notion of extended support in Section \ref{S:extended support}, in which we show that the thresholding gradient algorithm identifies this extended support after a finite number of iterations (Theorem \ref{T:finite identification of the extended support}).
In Section \ref{S:rates} we present some consequences of this result on the convergence of the algorithm.
We first derive in Section \ref{SS:general case} the strong convergence of the iterates, together with a general framework to guarantee rates.
We then specify our analysis to the function \eqref{E:LASSO} in Section \ref{SS:particular case}, and show the linear convergence of ISTA (Theorem \ref{T:CV linear for ISTA}).
We also consider in Section \ref{SS:L1Lp regularized LS} an elastic-net modification of \eqref{E:LASSO}, by adding an $\ell^p$ regularization term, and provide rates as well, depending on the value of $p \in ]1,+\infty[$.

\medskip

\section{Thresholding gradient methods}\label{S:threshold gradient method}

\paragraph{\textbf{Notation}}
We introduce some notation we will use throughout this paper.
$\mathcal{N}$ is a subset of $\mathbb{N}$. Throughout the paper, $X$ is a separable Hilbert space  endowed with the scalar product $\langle \cdot, \cdot \rangle$,
and $(e_k)_{k\in\mathcal{N}}$ is an orthonormal basis of $X$. 
Given $x \in X$, we set $x_k=\langle x,e_k \rangle$. The support of $x$ is $\supp(x)=\{\kinn \ | \ x_k \neq 0 \}$. Analogously, given  
$C\subset X$, $C_k=\{\langle x,e_k\rangle\,:\, x\in C\}$.
Given  $J \subset \N$,  the subspace supported by $J$ is denoted by $X_J=\{ x \in X \ | \ \supp(x) \subset J\}$ and  the subset 
of finitely supported vectors $c_{00}=\{x\in X\,:\, \supp(x) \text{ is finite }\}$. 
Given a collection of intervals $\{I_k\}_{k\in\mathcal{N}}$ of the real line, with a slight abuse of notation, we define, for every $k\in\mathcal{N}$,
\[
\mathbb{B}_{\infty,\mathcal{I}}=\bigoplus_{\kinn} I_k=\{x\in X\,:\, x=\sum_{k\in\mathcal{N}} t_k e_k, \text{ with } t_k\in I_k \text{ for every $k\in\mathcal{N}$} \}.
\]
Note that $\bigoplus_{k\in\mathcal{N}} {I}_k$ is a subspace of $X$. Therefore, the components of each element of $\bigoplus_{k\in\mathcal{N}} {I}_k$ must be  square summable.
The closed  ball of center $x\in X$ and radius $\delta\in\left]0,+\infty\right[$  is denoted by $\mathbb{B}_X(x,\delta)$.
Let $C\subset X$ be a closed convex set. Its indicator and support functions are denoted $\delta_C$ and $\sigma_C$, respectively, and the projection
onto $C$ is $\proj_C$.
Moreover, $\inte C$, $\bd C$, $\ri C$, and $\qri C$ will denote respectively the interior, the boundary, the relative interior, and the quasi relative interior of $C$ \cite[Section 6.2]{BauComV2}.
The set  of proper convex lower semi-continuous functions from $X$ to $\Rinf$ is denoted by $\Gamma_0(X)$.
Let $f\in\Gamma_0(X)$ and let $r\in\left]0,+\infty\right[$.  The sublevel set of $f$ is $S_f(r)=\{x \in X \ | \ f(x) - \inf f < r \}$.
The proximity operator of $f$ is defined as
\begin{equation*}
(\forall \lambda\in\left]0,+\infty\right[)\qquad \prox_{\lambda f}(x) = \argmin \{ y \in X \ | \ f(y) + \frac{1}{2 \lambda} \Vert y - x \Vert^2 \}.
\end{equation*}
Let $I\subset \mathbb{R}$ be a closed interval. Then, $\prox_{\sigma_{I}}=\soft_{I}$, where
	\begin{equation*}
(\forall t \in \R) \quad \soft_{I} (t) =
\begin{cases}
t - \inf I & \text{ if } t < \inf I \\
0 & \text{ if } t \in I \\
t - \sup I & \text{ if } t > \sup I,
\end{cases}
\end{equation*}
is the soft-thresholder corresponding to $I$.

\medskip

\paragraph{\bf Problem and main hypotheses}
We consider the general optimization problem
\begin{equation}
\label{E:Probleme Primal P}
\min\limits_{x \in X} \ f(x), \qquad f= g + h, \tag{$\text{\rm P}$}
\end{equation}
where typically $h$ will play the role of a smooth data fidelity term, and $g$ will be a nonsmooth sparsity promoting regularizer.
More precisely, we will make the following assumption:
\begin{equation}
\label{H:main structural assumption f=g+h}
\tag{$\text{\rm H}$}
\begin{cases}
\text{$h \in \Gamma_0(X)$ is bounded from below, }\\
\text{$h$ is differentiable, and $\nabla h$ is $L$-Lipschitz continuous on $X$, $L\in\left]0,+\infty\right[$},\\
\text{$g=\sum\limits_{\kinn}^{} g_k(\langle \cdot , e_k \rangle)$, with $g_k=\psi_k + \sigma_{I_k}$, where:} \\
\quad \bullet \, \text{for all $\kinn$, $I_k$ is a proper closed interval of $\R$, and $\I=\{I_k\}_\kinn$,} \\
\quad \bullet \,  \text{for all $\kinn$,}\, (\exists \omega > 0) \quad\text{ $[-\omega,\omega] \subset  I_k$,}\\
\quad \bullet \, \text{for all $\kinn$, $\psi_k \in \Gamma_0(\R)$ is differentiable at $0$ with $\psi_k(0)=0$ \text{ and } $\psi_k'(0)=0$.} 
\end{cases}
\end{equation}

\noindent As stated in the above assumption, in this paper we focus on  a specific class of functions $g$. They are given by 
the sum of a weighted $\ell^1$ norm and a positive smooth function minimized at the origin, namely:
$$\Vert \cdot \Vert_{1,\I} = \sum_{\kinn}^{} \sigma_{I_k},\quad\quad \psi = \sum_{\kinn}^{} \psi_k.$$ 
In \cite{ComPes07} the following characterization has been proved: the proximity operators of such functions $g$ 
are the monotone operators $T\colon X\to X$ such that for all $x\in X$, $T(x)=\left( T_k(x_k)\right)_\kinn$, for some $T_k\colon\mathbb{R}\to \mathbb{R}$ which satisfies
\[
(\forall \kinn)\quad T_k(x_k)=0 \iff x_k\in I_k.
\]
A few examples of such, so called,  \textit{thresholding operators} are shown in Figure \ref{F:prox thresholders}, and a more in-depth analysis can be found in \cite{ComPes07}.

\begin{center}
\begin{figure}[h]
\begin{minipage}{0.49\linewidth}
\begin{center}
\includegraphics[width=0.8\linewidth]{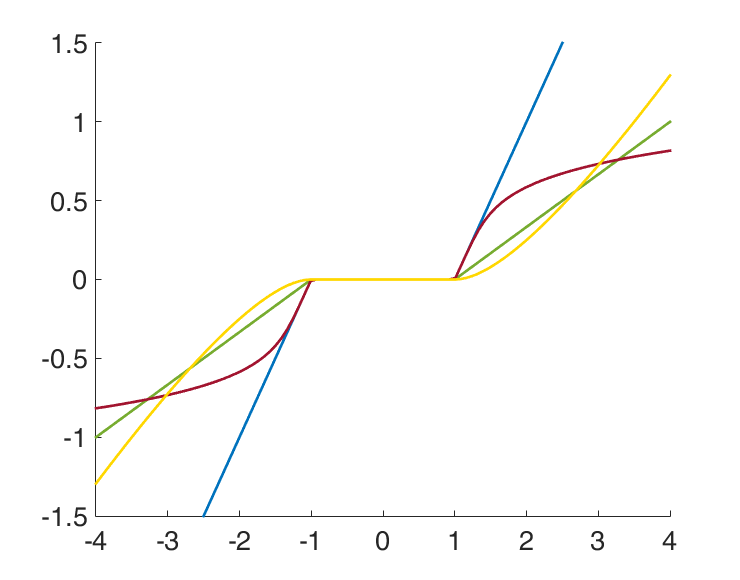}
\end{center}
\end{minipage}
\begin{minipage}{0.49\linewidth}
\begin{center}
\includegraphics[width=0.8\linewidth]{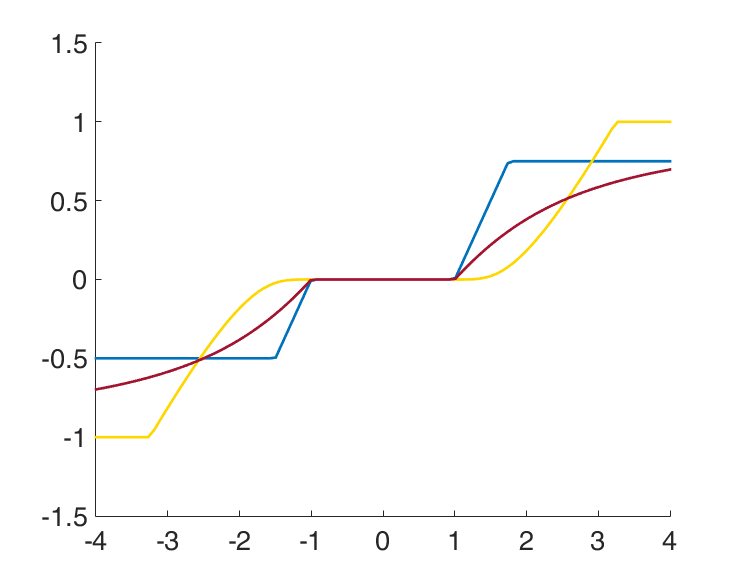}
\end{center}
\end{minipage}
\caption{\small Some examples of thresholding proximal operators in $\R$. 
On the left: $\vert \cdot \vert = \sigma_{[-1,1]}$ (blue), $\vert \cdot \vert + \vert \cdot \vert^{1.5}$ (yellow), $\vert\cdot \vert + \vert \cdot \vert^{2}$ (green), $\vert\cdot \vert + \vert \cdot \vert^{6}$ (red). 
On the right: $\vert \cdot \vert  + \delta_{[-0.5,0.75]}$ (blue), $ \vert \cdot \vert + \vert \cdot \vert^{1.5} + \delta_{[-1,1]}$ (yellow), and $\vert \cdot \vert - \ln(1- \vert \cdot \vert) + \delta_{]-1,1[}$ (red).
{ Observe that here the range of $\prox_g$ is equal to  the domain of $\partial \psi$.} }
\label{F:prox thresholders}
\end{figure}
\end{center}

A well-known approach to approximate solutions of \eqref{E:Probleme Primal P} is  the Forward-Backward algorithm  \cite{BauComV2}
\begin{equation}
\label{D:Forward backward}
x^0 \in X, \quad \lambda \in ]0, 2 L^{-1}[, \quad x^{n+1} = \prox_{\lambda g } (x^n - \lambda \nabla h(x^n)). \tag{FB}
\end{equation}
In our setting,  \eqref{D:Forward backward} is well-defined and specializes to a  \textit{thresholding gradient method}.
The Proposition below gathers some basic properties of $g$ and $f$ following from  assumption \eqref{H:main structural assumption f=g+h}.
\begin{proposition} The following hold. 
\label{P:simple facts about the problem}~~
\begin{enumerate}[(i)]
	\item\label{P:simple facts about the problem:L1 is support of infty ball} $\Vert \cdot \Vert_{1,\I}$ is the support function of $\B_{\infty,\I}=\bigoplus_{k\in\mathcal{N}} {I}_k$,
	\item\label{P:simple facts about the problem:domain is finite support seq} $\dom \partial \Vert \cdot \Vert_{1,\I}=c_{00}$,
	\item\label{P:simple facts about the problem:g is coercive} $g \in \Gamma_0(X)$ and it is coercive,
	\item\label{P:simple facts about the problem:solutions exist} $f$ is bounded from below and $\argmin f$ is nonempty,
	\item\label{P:simple facts about the problem:dual unique} the dual problem
\begin{equation}\label{E:Probleme Dual D}
\min\limits_{u \in X} \ g^*(u) + h^*(-u), \tag{$\text{\rm D}$}
\end{equation}
admits a unique solution $\bar u \in X$, and for all $\bar x \in \argmin f$, $\bar u = -\nabla h(\bar x)$.
	\item\label{P:simple facts about the problem:prox computation} for all $x \in X$ and all $\lambda > 0$, the proximal operator of $g$ can be expressed as
	\begin{equation*}
	\prox_{\lambda g } (x) = \sum_{k\in \mathcal{N}}  \prox_{\lambda \psi_k}\left(\soft_{\lambda I_k} (x_k)\right)  e_k .
	\end{equation*}
\end{enumerate}
\end{proposition}

\begin{proof} 
\ref{P:simple facts about the problem:L1 is support of infty ball}: see Proposition \ref{P:countable sum of functions}\ref{P:countable sum of support functions:coercivity and support box}.  

\ref{P:simple facts about the problem:domain is finite support seq}: see  Proposition \ref{P:countable sum of functions}\ref{P:countable sum of support functions:dom ssdif}.

\ref{P:simple facts about the problem:g is coercive}: see Proposition \ref{P:countable sum of functions}\ref{P:countable sum of support functions:coercivity and support box}.  

\ref{P:simple facts about the problem:solutions exist}: it is a  consequence of the coercivity of $g$ and the fact that both $h$ and $g$ are bounded from below.

\ref{P:simple facts about the problem:dual unique}: the smoothness of $h$ implies the strong convexity of $h^*$, and the existence and uniqueness of $\bar u$, see  \cite[Theorems 15.13 and 18.15]{BauComV2}.
The equality $\bar u = - \nabla h(\bar x)$ follows from \cite[Proposition 26.1(iv)(b)]{BauComV2}.

\ref{P:simple facts about the problem:prox computation}: it follows from \ref{P:countable sum of functions}\ref{P:countable sum of functions:prox} together with \cite[Proposition 3.6]{ComPes07}.
\end{proof}

\section{Extended support and finite identification}\label{S:extended support}

\subsection{Definition and basic properties}

We introduce the notion of extended support of a vector and prove some
basic properties of the support of solutions of problem \eqref{E:Probleme Primal P}.

\begin{definition}\label{D:support and extended support}
Let $x \in X$. The \textit{extended support} of $x$ is
\begin{equation*}
\esupp(x)= \supp(x) \cup \{ \kinn \ | \   -\nabla h(x)_k \in \bd I_k \}.
\end{equation*}
\end{definition}

It is worth noting that the notion of extended support depends on the problem \eqref{E:Probleme Primal P}, since its definition involves $h$ (see Remark \ref{R:extended support and active constraints} for more details).
It appears without a name in \cite{HalYinZha07}, and also in \cite{DegPeyFadJac16,Dos11,DuvPey17} for regularized least squares problems. 
Below we gather some results about the support and the extended support.

\begin{proposition}\label{P:extended support is finite}
Let $x \in \dom \partial f$, then $\supp(x)$ and $\esupp(x)$ are finite.
\end{proposition}

\begin{proof}
Let $x \in \dom \partial f = \dom \partial g$, and let $u\in \partial g(x)$ and let us start by verifying that $\supp(x)$ is finite.
Let $x^* \in \partial g(x)$, and let $y=x+x^*$.  Proposition \ref{P:simple facts about the problem}\ref{P:simple facts about the problem:prox computation} implies that for all 
$k \in \supp(x)$, $\prox_{\psi_k}\circ \soft_{I_k}(y_k) \neq 0$.
Lemma \ref{L:prox is injective at smooth minimizer} and the definition of $\soft_{I_k}$ imply that $y_k \notin  I_k$, and in particular that $\vert y_k \vert \geq  \omega$ for all $k \in \supp(x)$.  
Then we derive that
\begin{equation*}
\vert \supp(x) \vert = \omega^{-2} \sum\limits_{\footnotesize k \in \supp(x)}^{} \omega^2 \leq \omega^{-2} \sum\limits_{\footnotesize k \in \supp(x)}^{}  \vert y_k \vert^2 \leq \omega^{-2}	 \Vert y \Vert^2 < + \infty.
\end{equation*}
Next, we have to verify that $J$ is finite, where $J= \{ \kinn \ | \  -\nabla h(x)_k \in \bd I_k \}$.
If  $\mathcal{N}$ is finite, this is trivial.
Otherwise, we observe that $(\nabla h(x)_k)_\kinn \in \ell^2(\NN)$, which both implies that $\nabla h(x)_k$ tends to $0$ when ${k \to + \infty}$ in $\NN$.
Since $[-\omega,\omega] \subset I_k$, we deduce that $J$ must be finite.
\end{proof}

The following proposition clarifies the relationship between the support and the extended support for minimizers.

\begin{proposition}\label{P:relation between extended support and support} 
Let $\bar x \in \argmin f$. 
\begin{enumerate}[(i)]
	\item\label{P:relation between extended support and support:QC implication} If $0 \in \qri \partial f(\bar x)$ then $\esupp(\bar x)=\supp(\bar x)$.
\end{enumerate}
Assume that $\psi_k$ is differentiable on $\dom \partial \psi_k$,  for all $\kinn$. Then
\begin{enumerate}[resume*]
	\item\label{P:relation between extended support and support:QC} $\esupp(\bar x)=\supp(\bar x) \, \Leftrightarrow \, 0 \in \qri \partial f(\bar x)$.
\end{enumerate}
Assume moreover that $\psi \equiv 0$. Then
\begin{enumerate}[resume*]
	\item\label{P:relation between extended support and support:dual} $\esupp(\bar x) = \{ \kinn \ | \ -\nabla h(\bar x)_k \in \bd I_k \}$.
	\item\label{P:relation between extended support and support:independence} There exists ${J}\subset \mathcal{N}$ such that ${J}=\esupp(\bar x)$ for every $\bar x\in\argmin f$. 
	\item\label{P:relation between extended support and support:union of supports} $\esupp(\bar x)=\cup \{ \supp(x) \ | \ x \in \argmin f \} \, \Leftrightarrow \,  (\exists x \in \argmin f) \, \, 0 \in \qri \partial f(x)$.
\end{enumerate}
\end{proposition}
\begin{proof}[Proof of Proposition \ref{P:relation between extended support and support}]
Since $\bar x \in \argmin f \subset \dom \partial g$, it follows from Proposition \ref{P:extended support is finite} that $\supp(\bar x)$ is finite. Moreau-Rockafellar's sum rule \cite[Theorem 3.30]{Pey}, Proposition \ref{P:countable sum of functions}\ref{P:countable sum of functions:ssdiff}, Proposition \ref{P:derivative of support function}\ref{P:derivative of support function:0 gives full set} then yield
\begin{equation}\label{rbe1}
\partial f(\bar x)_k = \nabla h(\bar x)_k + \begin{cases} \partial \psi_k(\bar x_k) + \partial \sigma_{I_k}(\bar x_k) & \text{if } k\in\supp(\bar{x}) \\ {I}_k  & \text{if } k\notin \supp(\bar x).\end{cases}
\end{equation}
Since $\supp(\bar x)$ is finite and $\partial \psi_k(\bar x_k) + \partial \sigma_{I_k}(\bar x_k)$ is a closed interval of $\R$, Proposition \ref{P:qrinterior of infinite products} and Proposition \ref{P:derivative of support function}\ref{P:relative interiors:nonempty interior case} imply 
\begin{equation}\label{rbe2}
(\forall\kinn)\qquad\left(\qri \partial f(\bar x)\right)_{k} = \nabla h(\bar x)_k +\begin{cases} \ri (\partial \psi_k(\bar x_k) +  \partial \sigma_{I_k}(\bar x_k)) &\text{if } k \in \supp(\bar x) \\ 
\inte I_k &\text{if } k\notin \supp(\bar x).\end{cases} 
\end{equation}
\ref{P:relation between extended support and support:QC implication}: observe that
\begin{eqnarray}
0 \in \qri \partial f(\bar x) & \Rightarrow & (\forall k \notin \supp(\bar x)) \quad - \nabla h(\bar x)_k \in \inte I_k \label{rbe2.5}\\
& \Leftrightarrow & \{ \kinn \ | \ x_k = 0 \text{ and } - \nabla h(\bar x)_k \in\bd I_k \} = \emptyset \nonumber \\
& \Leftrightarrow & \esupp(\bar x) = \supp(\bar x). \nonumber
\end{eqnarray}
\ref{P:relation between extended support and support:QC}: note that from $0 \in \partial f(\bar x)$ and \eqref{rbe1}, we have $-\nabla h(\bar x)_k  \in \partial \psi_k(\bar x_k) + \partial \sigma_{I_k}(\bar x_k)$  for all $k \in \supp(\bar x)$.
But both $\psi_k$ and $\sigma_{I_k}$ are differentiable at $\bar x_k \neq 0$, so for all $k \in \supp(\bar x)$, $\qri (\partial \psi_k(\bar x_k) +  \partial \sigma_{I_k}(\bar x_k)) = \partial \psi_k(\bar x_k) +  \partial \sigma_{I_k}(\bar x_k) $ holds.
So we  deduce from \eqref{rbe2.5} that  item \ref{P:relation between extended support and support:QC} holds.

\ref{P:relation between extended support and support:dual}:
observe that, via \eqref{rbe1} and Proposition \ref{P:derivative of support function}\ref{P:derivative of support function:exposed faces in boundary},  for all $k \in \supp(\bar x)$, $-\nabla h(\bar x)_k \in \bd I_k$, meaning that indeed
$\esupp(\bar x) = \{ \kin \ | \ -\nabla h(\bar x)_k \in \bd I_k\}$.

\ref{P:relation between extended support and support:independence}: it follows from the uniqueness of $\nabla h(\bar x)$, see Proposition \ref{P:simple facts about the problem}\ref{P:simple facts about the problem:dual unique}.

\ref{P:relation between extended support and support:union of supports}: if  there is some $x \in \argmin f$ such that $0 \in \qri \partial f(x)$, we derive from  \ref{P:relation between extended support and support:QC} and 
\ref{P:relation between extended support and support:independence} that $\esupp(\bar x) = \supp(x)$.
So, the inclusion $\esupp(\bar x)\subset \cup \{ \supp(x') \ | \ x' \in \argmin f \}$ holds.
The reverse inclusion comes directly from the definition of $\esupp(\bar x)$ and  \ref{P:relation between extended support and support:independence}.
For the reverse inclusion, assume that $\esupp(\bar x)=\cup \{ \supp(x) \ | \ x \in \argmin f \}$ holds, and use the fact that $\esupp(\bar x)$ is finite to apply Lemma \ref{L:existence point support convex hull}, and obtain some $x \in \argmin f$ such that $\supp(x) = \esupp(\bar x)$.
We  then conclude that $0 \in \qri \partial f(x)$ using \ref{P:relation between extended support and support:independence} and \ref{P:relation between extended support and support:QC}.
\end{proof}

\begin{remark}[Extended support and active constraints]\label{R:extended support and active constraints}
Assume that $\psi = 0$.
Since   $g^*$ is the indicator function of $\B_{\infty,\I}$, in this case, the dual problem \eqref{E:Probleme Dual D} introduced in 
Proposition \ref{P:simple facts about the problem}\ref{P:simple facts about the problem:dual unique} 
can be rewritten as
\begin{equation}
\label{e:pdualsimple}
\min\limits_{\substack {u \in X\\ (\forall k\in\mathcal{N})\, u_k \in I_k}} \ h^*(-u). \tag{$\text{\rm D'}$}
\end{equation}
This problem admits a unique solution $\bar u \in \B_{\infty,\I}$, and the set of active  constraints at $\bar u$ is
\begin{equation*}
\{ \kinn \ | \ \bar u_k \in \bd I_k \}.
\end{equation*}
Since $\bar u=-\nabla h(\bar x)$ for any $\bar x \in \argmin f$ by Proposition \ref{P:simple facts about the problem}\ref{P:simple facts about the problem:dual unique}, Proposition \ref{P:relation between extended support and support}\ref{P:relation between extended support and support:dual} implies that the extended support for the solutions of \eqref{E:Probleme Primal P} is in that case nothing but the set of active constraints for the solution of \eqref{e:pdualsimple}.
\end{remark}

\begin{remark}[Maximal support and interior solution]
If $\psi = 0$ and the following (weak) qualification condition holds
\begin{equation}\label{E:weak CQ}
\tag{w-CQ} (\exists x \in \argmin f) \, \, 0 \in \qri \partial f(x),
\end{equation}
then, thanks to Lemma~\ref{L:existence point support convex hull} the extended support  is  the maximal support to be found among the solutions.
If for instance $h$ is the least squares loss on a finite dimensional space, it can be shown that the solutions having a maximal support are  the ones 
belonging to the relative interior of the solution set \cite[Theorem 2]{BarJouVai17}.
However, there are problems for which  \eqref{E:weak CQ} does not hold.
In such a case Proposition \ref{P:relation between extended support and support} implies that the extended support will be strictly larger than the maximal support (see Example \ref{Ex:noCQ}).
The gap between the maximal support and the extended support is equivalent to the lack of  duality between \eqref{E:Probleme Primal P} and \eqref{E:Probleme Dual D}.
\end{remark}

\begin{example}\label{Ex:CQ}
Let $g\colon\R^2\to\R:x\mapsto\|x\|_1$ and $h\colon\R^2\to\R\colon x\mapsto (x_1-x_2 -1)^2$.
In this case, $\argmin f = [\bar x^1,\bar x^2]$, where $\bar x^1=(0.5,0)$ and $\bar x^2 = (0,-0.5)$, as can be seen in Figure \ref{F:QC holds}.
The solutions $\bar x \in ]\bar x^1,\bar x^2[$ are the ones having the maximal support, since $\supp(\bar x) = \{1,2\}$, and also satisfy $0\in \rint \partial f(\bar x)$.
Instead, on the relative boundary of $\argmin f$ we have $\supp(\bar x^i) = \{i\}$ and $0 \notin \rint \partial f(\bar x^i)$ for $i\in\{1,2\}$.
This example is a one for which the extended support is the maximal support among the solutions.
\end{example}

\begin{example}\label{Ex:noCQ}
Let $g\colon\R\to\R:x\mapsto |x|$ and $h\colon\R\to\R\colon x\mapsto (x -1)^2/2$.
Then $\argmin f = \{\bar x\}$, with $\bar x=0$, as can be seen in Figure \ref{F:QC holds}.
The support of $\bar{x}$ is empty, and $0 \notin \rint \partial f (\bar x) = [-2,0]$.
In this case, condition \eqref{E:weak CQ} does not hold.
This can also be seen from the dual problem  $\min_{u \in [-1,1]} \ u^2/2-u$, whose unique constraint is active at the 
solution $\bar u = -\nabla h(\bar x) = 1$, meaning that $\esupp(\bar x) = \{1 \} \neq \supp(\bar x)$.
\end{example}

\begin{center}
\begin{figure}[h]

\begin{minipage}{0.3\linewidth}
\begin{center}
\includegraphics[width=0.95\linewidth]{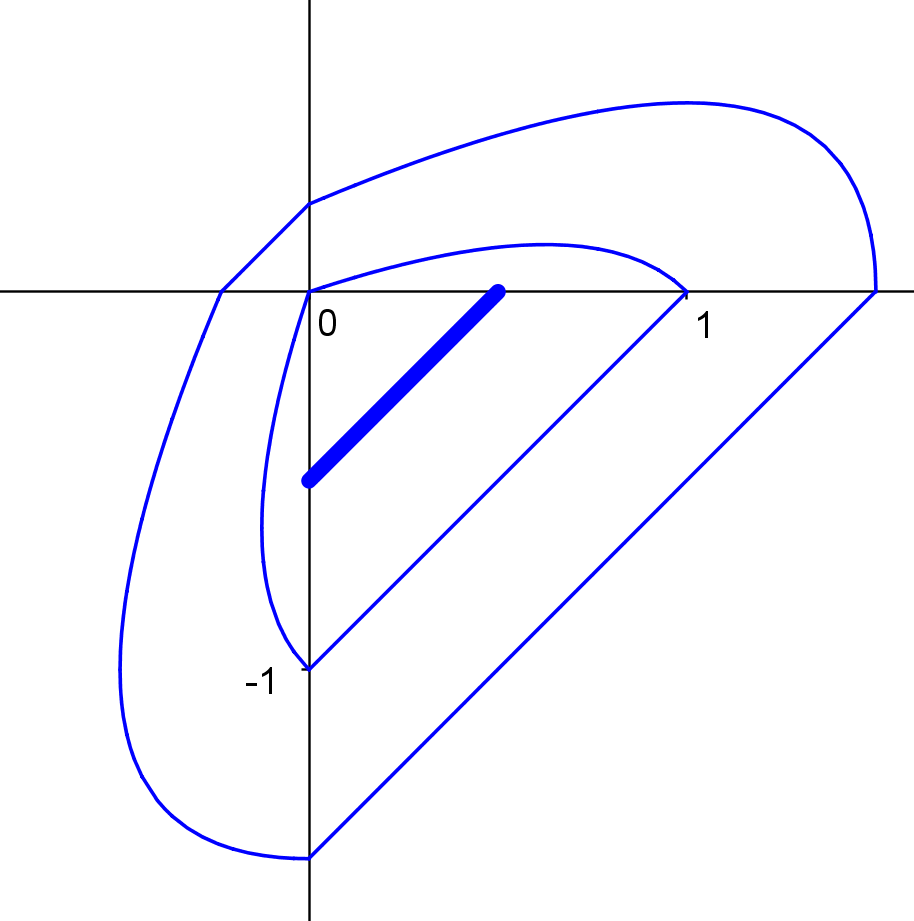}
\end{center}
\end{minipage}
\begin{minipage}{0.3\linewidth}
\begin{center}
\includegraphics[width=0.95\linewidth]{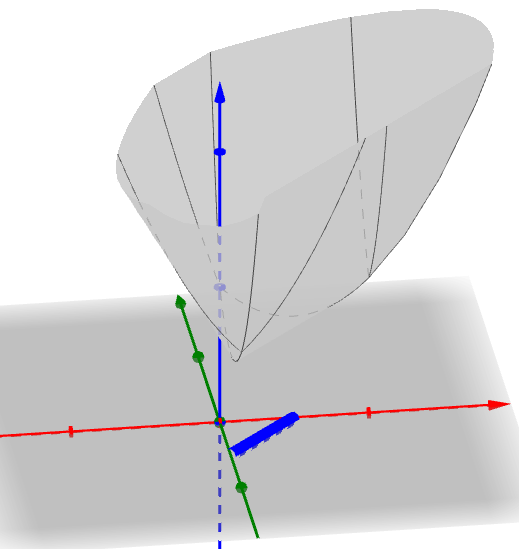}
\end{center}
\end{minipage}
\begin{minipage}{0.3\linewidth}
\begin{center}
\includegraphics[width=0.95\linewidth]{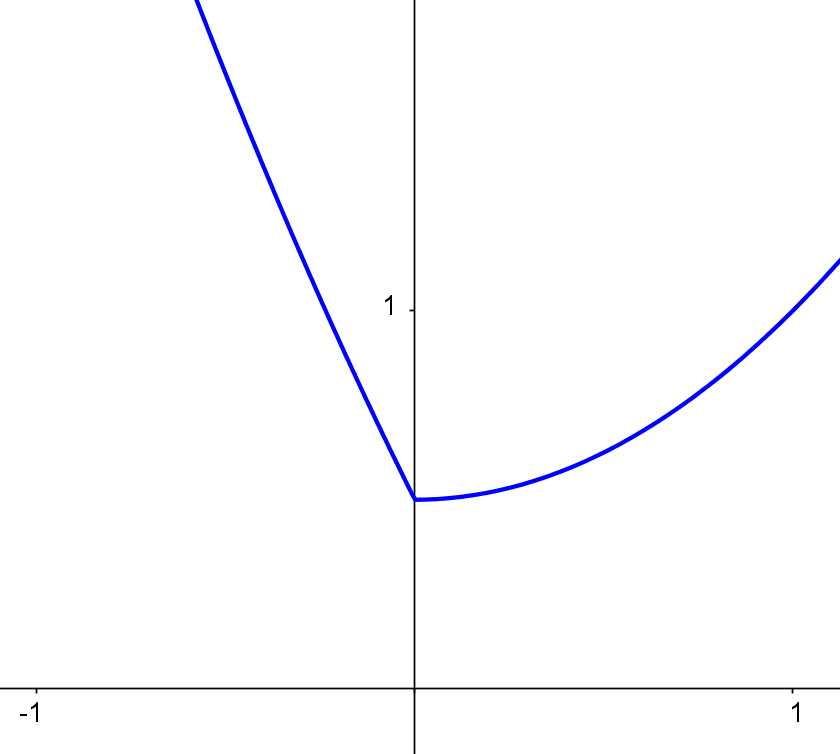}
\end{center}
\end{minipage}

\caption{Left and center: respectively level sets and graph of $f$ in Example \ref{Ex:CQ}, with $\argmin f$  in thick. Right: graph of $f$ in Example \ref{Ex:noCQ}.}\label{F:QC holds}
\end{figure}
\end{center}

\subsection{Finite identification}
A sparse solution $\bar{x}$ of problem \eqref{E:Probleme Primal P} is usually approximated by means of an iterative procedure $(x^n)_\nin$. To obtain
an interpretable approximation, a crucial property is that, after a finite number of iterations, the support of $x^n$ stabilizes 
and is included in the support of $\bar x$.
In that case, we say that the sequence $(x^n)_\nin$ \textit{identifies} $\supp(\bar x)$.
The  support identification property  has been the subject of active research in the last years \cite{HalYinZha07,Dos11,LiaFadPey14,FadMalPey17,DuvPey17}, and roughly speaking, 
 in finite dimension it is known that support identification holds whenever $\bar x$ satisfies the qualification condition $0 \in \qri \partial f(\bar x)$.
But this assumption is often not satisfied  in practice, in particular for noisy inverse problems (see e.g.  \cite{FadMalPey17}).
In \cite{HalYinZha07,DegPeyFadJac16},  the case $g(x) = \Vert x \Vert_1$  is studied in finite dimension and it is shown that  the extended support of $\bar x$ is identified even if the qualification condition does not hold. 
Thus, the  qualification condition $0 \in \qri \partial f(\bar x)$ is only used to ensure that the extended support coincides with the support (see Proposition \ref{P:relation between extended support and support}). 

In this section we extend these ideas to the setting of thresholding gradient methods in separable Hilbert spaces, and we show in Theorem \ref{T:finite identification of the extended support} that indeed the 
extended support is always identified after a finite number of iterations.
For this, we  need to introduce a quantity, which measures the stability of the dual problem \eqref{E:Probleme Dual D}.
\begin{definition}\label{D:degree of degeneracy}
We define the function $\rho \colon X \longrightarrow \R$ as follows:
\begin{equation}
(\forall u \in X) \quad \rho(u)= \inf\limits_{u_k \in \inte I_k} \dist(u_k,\bd I_k).
\end{equation}
Also, given any $\bar x \in \argmin f$, we define $\rho_{\rm sol}=\rho(-\nabla h(\bar x))$.
\end{definition}

It can be verified that $\rho(u) \in \left]0,+\infty\right[$ for all $u \in X$ (this is left in the Annex, see Proposition~\ref{P:interior of infinite products}). 
Moreover, $\rho_{\rm sol}$ is uniquely defined, thanks to Proposition \ref{P:simple facts about the problem}\ref{P:simple facts about the problem:dual unique}.

\begin{theorem}[Finite identification of the extended support]\label{T:finite identification of the extended support}
Let $(x^n)_\nin$ be generated by the Forward-Backward algorithm \eqref{D:Forward backward}, and let $\bar x$ be any minimizer of $f$.
Then,  the number of iterations for which the support of $x^n$ is not included in $\esupp(\bar x)$ is finite, and cannot exceed $ \rho_{\rm sol}^{-2}\lambda^{-2} \Vert x^0 - \bar x \Vert^2 $.
\end{theorem}

\begin{remark}[Optimality of the identification result]
Theorem \ref{T:finite identification of the extended support} implies that after some iterations the inclusion $\supp(x^n)\subset\esupp(\bar x)$ holds.
Let us verify that it is impossible to improve the result,  i.e. that in general we cannot identify a set smaller than $\esupp(\bar x)$.
In other words, is it true that
\begin{equation}
\label{e:smartinit}
(\exists x^0 \in X)(\exists \bar x \in \argmin f)(\forall \nin) \quad \supp(x^n)=\esupp(\bar x)?
\end{equation}
If \eqref{E:weak CQ} holds, the answer is yes.
Indeed, if there is $\bar x \in \argmin f$ such that $0\in \qri \partial f(\bar x)$, we derive from Proposition \ref{P:relation between extended support and support}\ref{P:relation between extended support and support:QC implication} that $\esupp(\bar x) = \supp(\bar x)$. 
So by taking $x^0 = \bar x$, and using the fact that it is a fixed point for the Forward-Backward iterations, we conclude that $\supp(x^n) \equiv \esupp(\bar x)$.
If \eqref{E:weak CQ} does not hold, then this argument cannot be used, and it is not clear in general if there always exists an 
initialization which produces a sequence verifying \eqref{e:smartinit}.
Consider for instance the function in Example \ref{Ex:noCQ}.
Taking $x^0 \in ]0,+\infty[$ and a stepsize $\lambda \in ]0,1[$,  the iterates are defined by $x^{n+1} = (1-\lambda)x^n$, meaning that for all $\nin$, $\supp(x^n) \equiv \{1\}$, which is exactly $\esupp(\bar x)$.
So in that case \eqref{e:smartinit} holds true.
\end{remark}

\begin{proof}
Let $\bar x \in \argmin f$, and let $E=X_{\esupp(\bar x)}$ be the finite dimensional subspace of $X$ supported by  $\esupp(\bar x)$.
First define the ``gradient step'' operator
\begin{equation*}
T_{\lambda h}=\Id-\lambda \nabla h,
\end{equation*}
so that the Forward-Backward iteration can be rewritten as $x^{n+1}= \prox_{\lambda g}(T_{\lambda h}(x^n))$.
Proposition \ref{P:simple facts about the problem}\ref{P:simple facts about the problem:prox computation} implies that for all $\kinn$ and all $n \in \N^*$, 
\begin{equation}\label{ifi01}
x_k^n=\prox_{\lambda \psi_k} \circ \soft_{\lambda I_k}(T_{\lambda h}(x^{n-1 })_k).
\end{equation}
Since $\bar x$ is a fixed point for the forward-backward iteration \cite[Proposition 26.1(iv)]{BauComV2}, we also have 
\begin{equation}\label{ifi02}
\bar x_k=\prox_{\lambda \psi_k} \circ \soft_{\lambda I_k}(T_{\lambda h}(\bar x)_k).
\end{equation}
Using the fact that $\prox_{\lambda \psi_k}$ is nonexpansive, and that $\soft_{\lambda I_k}$ is firmly non-expansive  \cite[Proposition 12.28]{BauComV2}, we derive
\begin{eqnarray*}
\Vert x^n - \bar x \Vert^2 & = & \sum\limits_\kinn \vert x_k^n - \bar x_k \vert^2 \ 
\leq 
\ \sum\limits_\kinn \vert \soft_{\lambda I_k}(T_{\lambda h}(x^{n-1 })_k) - \soft_{\lambda I_k}(T_{\lambda h}(\bar x)_k) \vert^2 \\ 
& \leq & \sum\limits_\kinn \vert T_{\lambda h}(x^{n-1 })_k - T_{\lambda h}(\bar x)_k \vert^2 - \vert (\Id-\soft_{\lambda I_k})(T_{\lambda h}(x^{n-1 })_k) - (\Id-\soft_{\lambda I_k})(T_{\lambda h}(\bar x)_k) \vert^2 \\
& \leq & \Vert T_{\lambda h}(x^{n-1 }) - T_{\lambda h}(\bar x) \Vert^2  - \sigma_{n,{k}}^2,
\end{eqnarray*}
where
\begin{equation*}
\sigma_{n,{k}}=\vert (\Id-\soft_{\lambda I_k})(T_{\lambda h}(x^{n-1 })_{k}) - (\Id-\soft_{\lambda I_k})(T_{\lambda h}(\bar x)_{k}) \vert.
\end{equation*}
Moreover, the gradient step operator  $T^G$ is  non-expansive since $\lambda\in\left]0,2L^{-1}\right[$ (see e.g. \cite[Lemma 3.2]{Lem96}), so we end up with
\begin{equation}\label{ifi0}
(\forall n \in \N^*)(\forall \kinn) \quad \Vert x^n - \bar x \Vert^2 \leq \Vert x^{n-1} - \bar x \Vert^2 - \sigma_{n,{k}}^2.
\end{equation}
The key point of the proof is to get a nonnegative lower bound for $\sigma_{n,{k}}$ which is independent of $n$, when $x^n \notin E$.

Assume  that there is some $\nin^*$  such that $x^n \notin E$.
This means that there exists $\kinn \setminus \esupp(\bar x)$ such that $x^n_{k} \neq 0$.
Also, since $\supp(\bar x) \subset \esupp(\bar x)$, we must have $\bar x_{k} =0$, meaning that $T_{\lambda h}(\bar x)_k = - \lambda \nabla h(\bar x)_k$.
 We deduce from \eqref{ifi01}, \eqref{ifi02}, and Lemma \ref{L:prox is injective at smooth minimizer}, that
\begin{equation}\label{ifi03}
T_{\lambda h}(x^{n-1 })_k \notin \lambda I_k \text{ and } T_{\lambda h}(\bar x)_k \in \inte \lambda I_k.
\end{equation}
Since $\Id-\soft_{\lambda I_k}$ is the projection on $\lambda I_k$, we derive from \eqref{ifi03} that
\begin{equation*}
\sigma_{n,{k}}=\vert \proj_{\lambda I_k} (T_{\lambda h}(x^{n-1 })_{k}) - T_{\lambda h}(\bar x)_{k} \vert.
\end{equation*}
Moreover $\proj_{\lambda I_k} (T_{\lambda h}(x^{n-1 })_{k}) \in \bd I_k$, therefore by Definition \ref{D:degree of degeneracy} and \eqref{ifi03}, we obtain that
\begin{equation*}
\sigma_{n,{k}} \geq  \lambda \dist(\lambda^{-1} T_{\lambda h}(\bar x)_k, \bd I_k) \geq \lambda \rho(\lambda^{-1}T_{\lambda h}(\bar x)_k) = \lambda\rho(-\nabla h(\bar x)_k) = \lambda  \rho_{\rm sol}.
\end{equation*}
Plugging this into \eqref{ifi0}, we obtain
\begin{equation}\label{ifi05}
\forall n \in \N^*, \ x^n \notin E \ \Rightarrow \ \Vert x^n - \bar x \Vert^2 \leq \Vert x^{n-1} - \bar x \Vert^2 -     \rho_{\rm sol}^2 \lambda^2  .
\end{equation}
Next note that the sequence $(x^n)_{\nin}$ is F\'ejer monotone with respect to the minimizers of $f$ (see e.g. \cite[Theorem 2.2]{GarRosVil17}) --- meaning that $(\Vert x^n - \bar x \Vert)_\nin$ is a decreasing sequence. Therefore the inequality \eqref{ifi05} cannot hold an infinite number of times.
More precisely, $x_n \notin E$ can hold for at most  $ \lambda^{-2}  \rho_{\rm sol}^{-2}\Vert x^0 - \bar x \Vert^2 $ iterations.
\end{proof}

\section{Strong convergence and rates}\label{S:rates}

\subsection{General results for thresholding gradient methods}\label{SS:general case}

Strong convergence of the iterates for the thresholding gradient algorithm was first stated in \cite[Section 3.2]{DauDefDem04} for $g=\|\cdot\|_1$, and then
generalized to general thresholding gradient methods in \cite[Theorem 4.5]{ComPes07}.
We provide a new and simple proof for this result, exploiting the "finite-dimensionality" provided by the  identification result in 
Theorem \ref{T:finite identification of the extended support}.

\begin{corollary}[Finite dimensionality for thresholding gradient methods]\label{T:strong convergence for soft threshold gradient}
Let $(x^n)_\nin$ be generated by a thresholding gradient algorithm. 
Then:
\begin{enumerate}[(i)]
	\item\label{T:strong convergence for soft threshold gradient:finite dimensionality} There exists a finite set $J \subset \mathcal{N}$ such that $x^n \in X_J$ for all $n \in \N^*$.
	\item\label{T:strong convergence for soft threshold gradient:strong convergence} $x^n$ converges strongly to some $\bar x \in \argmin f$.
\end{enumerate}
\end{corollary}

\begin{proof}
\ref{T:strong convergence for soft threshold gradient:finite dimensionality}: let $x \in \argmin f$ and let
\begin{equation*}
J= \esupp (x)  \bigcup \{ \supp(x^n) \ | \ \nin^*, x^n \notin X_{\footnotesize \esupp(x)} \},
\end{equation*}
and observe that it is finite, as a finite union of finite sets (see Proposition \ref{P:extended support is finite} and Theorem \ref{T:finite identification of the extended support}).

\ref{T:strong convergence for soft threshold gradient:strong convergence}: it is well known that $\argmin f \neq \emptyset$ implies  that $(x^n)_{\nin}$ converges weakly towards 
some $\bar x \in \argmin f$ (see e.g. \cite[Theorem 2.2]{GarRosVil17}).
In particular, $(x^n)_{\nin}$ is a bounded sequence in $X$.
Moreover, \ref{T:strong convergence for soft threshold gradient:finite dimensionality} implies that $(x^n)_{n \in \N^*}$ belongs to $X_J$, which is finite dimensional.
This two facts imply that $(x^n)_{n \in \N^*}$ is contained in a compact set of $X$ with respect to the strong topology, and thus converges strongly.
\end{proof}

Next we discuss the rate of convergence for the thresholding gradient methods.
Beforehand, we briefly recall how the geometry of a function around its minimizers is related to the rates of convergence of the Forward-Backward algorithm.

\begin{definition}\label{D:conditioning}
Let $p \in [2,+\infty[$ and $\Omega \subset X$.
We say that $\phi \in \Gamma_0(X)$ is $p$-conditioned on $\Omega$ if
\begin{equation*}
(\exists \gamma_{\phi,\Omega} >0)(\forall x \in \Omega) \quad \frac{\gamma_{\phi,\Omega}}{p}\dist(x,\argmin \phi)^p \leq \phi(x) - \inf \phi.
\end{equation*}
\end{definition}

\noindent A $p$-conditioned function is a function which somehow behaves like $\dist(\cdot,\argmin \phi)^p$ on a set.
For instance, strongly convex functions are $2$-conditioned on $\Omega = X$, and the constant $\gamma_{\phi,X}$ is nothing but the constant of strong convexity.
But the notion of $p$-conditioning is more general and also describes the geometry of functions having more than one minimizer. 
For instance in finite dimension, any  positive quadratic function is $2$-conditioned on $\Omega = X$, in which case the constant $\gamma_{\phi,X}$ is  the smallest \textit{nonzero} eigenvalue of the hessian.
This notion is  interesting since it allows to get precise convergence rates for some algorithms (including the Forward-Backward one) \cite{AttBolSva13}:
\begin{itemize}
	\item sublinear rates if $p >2$,
	\item linear rates if $p=2$.
\end{itemize}
For more examples, related notions and references,  we refer the interested reader to  \cite{DruLew16,BlaBol16,GarRosVil17}.

Corollary \ref{T:strong convergence for soft threshold gradient}   highlights the fact that the behavior of the thresholding gradient method 
 essentially depends on the conditioning of $f$ on finitely supported subspaces.
It is then natural to introduce the following notion of finite uniform conditioning.

\begin{definition}
Let $p \in [2, +\infty[$. We say that a function $\phi \in \Gamma_0(X)$ satisfies the \textit{finite uniform conditioning property of order }$p$ if, 
for every finite $J \subset \mathcal{N}$, $\forall \bar x \in \argmin \phi$, $ \forall (\delta,r) \in ]0,+\infty[^2$, $\phi$ is $p$-conditioned on $X_J \cap \B_X(\bar x,\delta) \cap S_\phi(r)$.
\end{definition}

\begin{remark}\label{R:FUG trivial if no minimizers}
In this definition, we only need information about $\phi$ over supports $J$ satisfying $\argmin \phi \cap X_J \neq \emptyset$.
Indeed, if $\argmin \phi \cap X_J = \emptyset$, then $\phi$ is $p$-conditioned on  $X_J \cap \B_X(\bar x,\delta) \cap S_\phi(r)$ for 
any $(\delta,r)$ and for all $p \in [2,+\infty[$ according to \cite[Proposition 3.4]{GarRosVil17}.
\end{remark}

In the following theorem, we illustrate how  finite uniform conditioning  guarantees global rates of convergence for the threshold gradient methods: linear rates if $p=2$, and sublinear rates for $p>2$.
Note that these sublinear rates are better than the $O(n^{-1})$ rate guaranteed in the worst case.

\begin{theorem}[Convergence rates for threshold gradient methods]\label{T:CV thresholding gradient method} 
Let $(x^n)_\nin$ be generated by the Forward-Backward algorithm \eqref{D:Forward backward}, and let $\bar x \in \argmin f$ be its (weak) limit. Then the following hold.
\begin{enumerate}[(i)]
\item \label{T:CV thresholding gradient method i}
 If $f$ satisfies the finite uniform conditioning property of order $2$, then there  exist  $\eps \in \left]0,1\right[$ and $C \in\left]0,+\infty\right[$, depending on $(\lambda,f,x^0)$, such that
\begin{equation*}
(\forall n \geq 1) \quad f(x^n) - \inf f \leq \eps^n (f(x^0)-\inf f) \quad \text{ and } \quad  \Vert x^{n+1} - \bar x \Vert \leq C \sqrt{\eps}^n.
\end{equation*}
\item \label{T:CV thresholding gradient method ii}
If $f$ satisfies the finite uniform conditioning property of order $p>2$, then  there  exist  $(C_1,C_2)\in\left]0,+\infty\right[^2$, depending on $(\lambda,f,x^0)$, such that
\begin{equation*}
(\forall n \geq 1) \quad f(x^n) - \inf f \leq C_1{n^{-\frac{p}{p-2}}} \quad \text{ and }\quad  \Vert x_{n+1} - x_\infty \Vert \leq  C_2{n^{-\frac{1}{p-2}}}.
\end{equation*}
\end{enumerate}
\end{theorem}

\begin{proof}
According to Corollary \ref{T:strong convergence for soft threshold gradient}, there exists a finite set $J \subset \mathcal{N}$ such that for all $n \geq 1$, $x^n \in X_J$, and $x^n$ converges strongly to $\bar x \in \argmin f$.
Also, the decreasing and F\'ejer properties of the Forward-Backward algorithm (see e.g. \cite[Theorem 2.2]{GarRosVil17}) tells us that for all $\nin$, $x^n \in \B_X(\bar x,\delta) \cap S_f(r)$, by taking $\delta=\Vert x^0 - \bar x \Vert$ and $r= f(x^0) - \inf f$.
Therefore, thanks to the finite uniform conditioning assumption, we can apply \cite[Theorem 4.2]{GarRosVil17} to the sequence $(x^{n+1})_\nin \subset \Omega=X_J \cap \B_X(\bar x,\delta) \cap S_f(r)$ and conclude.
\end{proof}

\subsection{$\ell^1$ regularized least squares}\label{SS:particular case}
Let $A\colon X\to Y$ be a linear operator from $X$ to a separable Hilbert space $Y$, and let $y \in Y$.
In this section, we discuss the particular case when $h(x)=\frac{1}{2}\Vert Ax - y \Vert_Y^2$ and 
$\psi \equiv 0$. The function in \eqref{E:Probleme Primal P} then becomes
\begin{equation*}
X\ni x\mapsto f(x) = \Vert x \Vert_{1,\I} + \frac{1}{2}\Vert Ax-y\Vert_Y^2,
\end{equation*}
and the Forward-Backward algorithm specializes to the iterative soft-thresholding algorithm (ISTA).
In this special case, linear convergence rates have been studied under additional assumptions on
the operator $A$. A common one is injectivity of $A$ or, more generally,  the so-called Finite Basis Injectivity property (FBI) \cite{BreLor08b}.
The FBI requires $A$ to be injective once restricted to $X_J$, for any finite $J \subset \mathcal{N}$.
It is clear that the FBI property implies that $h$  is a strongly convex function once restricted to each $X_J$, meaning that the finite uniform conditioning of order $2$ holds.
So, the linear rates obtained in \cite[Theorem 1]{BreLor08b} under the FBI assumption can be directly derived from Theorem \ref{T:CV thresholding gradient method}.
However, as can be seen  in Theorem \ref{T:CV thresholding gradient method} , strong convexity is not necessary to get linear rates, and the finite uniform $2$-conditioning is 
 a sufficient condition (and it is actually  necessary, see \cite[Proposition 4.18]{GarRosVil17}).
By using Li's Theorem on convex piecewise polynomials \cite[Corollary 3.6]{Li13}, we show in Proposition \ref{P:LASSO is FUG} below  that 
$f$ satisfies a  finite uniform conditioning of order $2$ on finitely supported subsets,  without doing \textit{any assumption} on the problem.
First, we need a technical Lemma which establishes the link between the conditioning of a function on a finitely supported space and the conditioning of its restriction to this space.

\begin{lemma}\label{L:finite to infinite conditioning}
Let $\phi \in \Gamma_0(X)$, let $m\in\N^*$ and let $J=\{k_1,\ldots,k_{m}\} \subset \mathcal{N}$.  Suppose that $\bar x \in \argmin \phi \cap X_J$.  Let
$\Xi: \R^{m} \rightarrow X_J\colon (u_1,...,u_m)\mapsto\sum_{i=1}^{{m}} u_i e_{k_i}$.
 Assume that, for every $ (\delta,r) \in ]0,+\infty[^2$,
\begin{equation*}
\text{$\phi_J = \phi \circ \Xi\in \Gamma_0(\R^{m})$ is $p$-conditioned on 
$\B_{\R^{m}}(\Xi^{-1}(\bar{x}),\delta)\cap S_{\phi_J}(r)$}
\end{equation*}
Then  $\phi$ is $p$-conditioned on $X_J$.
\end{lemma}

\begin{proof}
Assume without loss of generality that $k_1<...<k_m$.
Also, observe that $\bar x \in {X_J} \cap \argmin \phi$ implies that $\bar{u}=\Xi^{-1}(\bar x)$ is well-defined.
By definition, $\inf \phi \leq \inf \phi_J$, and
\begin{equation*}\label{pli3}
\inf \phi= \phi(\bar x) = \phi \circ \Xi(\bar u) = \phi_J(\bar u) \geq \inf \phi_J,
\end{equation*}
which implies  $\inf \phi = \inf \phi_J$.
Also, we have
\begin{equation*}
x \in \Xi(\argmin \phi_J) \Leftrightarrow 
x = \Xi(u) \text{ and } \phi_J(u) = \inf \phi_J
\Leftrightarrow
x \in X_J \text{ and } \phi(x) = \inf \phi,
\end{equation*}
meaning that $\Xi(\argmin \phi_J) = \argmin \phi \cap {X_J}.$
Let $(\delta,r)\in\left]0,+\infty\right[^2$, and let $\Omega=X_J \cap \B_X(\bar x,\delta) \cap S_\phi(r)$.
Since $\phi_J$ is $p$-conditioned on $\B_{\R^{m}}(\Xi^{-1}(\bar{x}),\delta)\cap S_{\phi_J}(r)$ 
there exists  $\gamma\in\left]0,+\infty\right[$ such that
\begin{equation}\label{liq1}
(\forall u \in \B_{\R^m}(\bar u,\delta)\cap S_{\phi_J}(r)) \quad \frac{\gamma}{p}\dist(u,\argmin \phi_J)^{p} \leq \phi_J(u) - \inf \phi_J.
\end{equation}
Let $x=\Xi(u)$ in \eqref{liq1}. Since $\|\Xi\|=1$, it is easy to see that $\Vert x - \bar x \Vert \leq \delta$ and $\phi(x) - \inf \phi < r$.
So we can rewrite \eqref{liq1} as:
\begin{equation*}\label{liq2}
(\forall x \in \Omega) \quad \frac{\gamma}{p}\dist(\Xi^{-1}x,\argmin \phi_J)^{p} \leq \phi(x) - \inf \phi.
\end{equation*}
It follows from $\Xi(\argmin \phi_J) = \argmin \phi \cap {X_J}$ that
\begin{equation*}\label{liq3}
(\forall x \in \Omega) \quad \phi(x) - \inf \phi 
\geq \frac{\gamma}{p} \dist(x, \argmin \phi \cap X_J)^{p} 
\geq \frac{\gamma}{p}  \dist(x, \argmin f)^{p} .
\end{equation*}
Therefore $\phi$ is ${p}$-conditioned on $\Omega$.
\end{proof}

\begin{proposition}[Conditioning of $\ell^1$ regularized least squares]
\label{P:LASSO is FUG}
Let $(Y,\|\cdot\|_Y)$ be  a separable Hilbert space, let $y\in Y$ and let $A\colon X\to Y$ be a bounded linear operator.
In assumption \eqref{H:main structural assumption f=g+h} suppose that for every $\kinn$, $I_k \in \I$ is bounded.
Then $X\ni x\mapsto f(x)=\Vert x \Vert_{1,\I} + \frac{1}{2}\Vert Ax-y\Vert_Y^2$ has a finite uniform conditioning of order $2$.
\end{proposition}
\begin{proof}
Let $J \subset \mathcal{N}$, $J=\{k_1,\ldots,k_m\}$, with $k_1<\ldots<k_m$, and suppose that $\argmin f\cap X_J\neq\varnothing$. Define, using the same notation as in Lemma \ref{L:finite to infinite conditioning}
\begin{equation*}\label{pli1}
h_J\colon\mathbb{R}^m\to \mathbb{R}\colon u\mapsto\frac{1}{2}\Vert A\Xi u - y \Vert_Y^2.
\end{equation*}
Define $A_J=A\Xi: \R^m \to Y$, and let $S_J=(A^*_J A_J)^{1/2}$, which verifies $R(S_J^*) = R(A^*_J)$.
Thus, there exists $y_J \in \R^m$ such that $A^*_J y = S^*_J y_J$, so that we can rewrite
\begin{equation}\label{pli2}
h_J(u) = \frac{1}{2}\Vert A_J u \Vert^2_{\R^m} + \frac{1}{2} \Vert y \Vert_Y^2 - \langle A_J u,y \rangle_Y = \frac{1}{2} \Vert S_J u - y_J \Vert^2_{\R^m} + \frac{1}{2} (\Vert y \Vert^2_Y - \Vert y_J\Vert^2_{\R^m}).
\end{equation}
Set $s_k=S_Je_k \in Y$. Then, \eqref{pli2} yields
\begin{equation*}
f_J(u) =
 \sum\limits_{i=1}^{m} \sigma_{I_{k_i}}(u_i) 
 +\frac{1}{2}\sum\limits_{i,j=1}^{m} \langle s_{k_i}, s_{k_j} \rangle_Y u_i u_j 
 - \sum\limits_{i=1}^{m} {(S_J^*y_J)}_{i} u_i  
 + \frac{1}{2}\Vert y \Vert^2_Y.
\end{equation*}
Since the intervals $I_k$ are bounded, their support functions are finite valued and piecewise linear, so $f_J$ is a piecewise polynomial of degree two in $\R^m$.
We  then apply \cite[Corollary 3.6]{Li13} to derive that $f_J$ is $2$-conditioned on $S_{f_J}(r)$, for any $r\in\left]0,+\infty\right[$.
We conclude by using Lemma \ref{L:finite to infinite conditioning}.
\end{proof}

 Combining Theorem \ref{T:CV thresholding gradient method} and Proposition \ref{P:LASSO is FUG}, we can now  state our main result concerning the linear rates of ISTA. 

\begin{theorem}[Linear convergence for the iterative soft thresholding]\label{T:CV linear for ISTA}
Under the assumptions of Proposition~\ref{P:LASSO is FUG},  let $(x^n)_\nin$ be  the sequence generated by the forward-backward algorithm applied 
to $ f$.
Then $(x^n)_\nin$ converges strongly to some $\bar x \in \argmin f$, and there exists two constants $\eps \in \left]0,1\right[$ and $C \in\left]0,+\infty\right[$, depending on $(\lambda,L,x^0,\I,A,y)$, such that
\begin{equation*}
(\forall n \geq 1) \quad f(x^n) - \inf f \leq \eps^n (f(x^0)-\inf f) \quad \text{ and } \quad  \Vert x^{n+1} - \bar x \Vert \leq C \sqrt{\eps}^n.
\end{equation*}
\end{theorem}

\begin{remark}[On the linear rates]
The convergence rate for the iterative soft-thresholding has been a subject of interest since years, and have been obtained
only under additional assumptions on $A$ \cite{BreLor08b}. Theorem \ref{T:CV linear for ISTA} closes the question about the linear rates, by proving that they always hold. 
However, there are still several open problems, related to the  estimation of the constant appearing in these linear rates.
This  is related to the estimation of the constant $\gamma_{f,\Omega}$ in Definition \ref{D:conditioning}, when $\Omega=S_f(r) \cap X_J$ for some finite $J\subset \mathcal{N}$.
Up to now, the only available result  is based on Hoffman's lemma, which doesn't allow for explicit lower bounds on $\gamma_{f,\Omega}$ \cite{NecNesGli15,BolNguPeySut16}.
Having a tight lower bound for $\gamma_{f,\Omega}$, depending on $A$ restricted to $X_J$, would be of interest to go in this direction.
\end{remark}

\subsection{$\ell^1+\ell^p$ regularized least squares}\label{SS:L1Lp regularized LS}

We are now interested in   $\ell^1+\ell^p$-regularizers, i.e. when 
\begin{equation*}
g(x)= \Vert x \Vert_{1,\I} + \frac{1}{p}\Vert x \Vert^p, \text{ with } \Vert x \Vert^p=\sum\limits_{\kin}^{} \vert x_k \vert^p, \quad p> 1.
\end{equation*}
The case $p=2$ is also known as \textit{{elastic net regularization}} and has been proposed in \cite{ZouHas05}. The elasitc-net
penalty has been studied by the statistical machine learning community as an alternative to the $\ell^1$ regularization in variable selection problems where there 
are highly correlated features and all the relavant ones have to be identified \cite{DemDevRos09}. See also \cite{ComSalVil15} for the case $p<2$.
Note that the proximal operator of such $g$ can be computed explicitly when $p\in \{4/3,3/2,2,3,4\}$ (see \cite{ComPes07}).

\begin{proposition}[Geometry of ($\ell^1+\ell^p$) regularized least squares]\label{P:p-LASSO is FUG}
Let $p\in ]1,+\infty[$, let $(Y,\|\cdot\|_Y)$ be  a separable Hilbert space, let $y\in Y$ and let $A\colon X\to Y$ be a bounded linear operator.
In assumption \eqref{H:main structural assumption f=g+h} suppose that for every $\kinn$, $I_k \in \I$ is bounded.
Then $f\colon X\to \mathbb{R}\colon x\mapsto\Vert x \Vert_{1,\I} + \frac{1}{p}\Vert x \Vert_p^p + \frac{1}{2}\Vert Ax-y\Vert_Y^2$ has a finite uniform conditioning of order $\max\{2,p\}$.
\end{proposition}

\begin{proof}
Let $J \subset \mathcal{N}$, $m=\vert J\vert$ and $p'=\max\{p,2\}$.
We define, by using the same notation as in Lemma \ref{L:finite to infinite conditioning},
\begin{equation*}\label{pli1b}
g_J(u)=\sum\limits_{i=1}^{m} \sigma_{I_{k_i}}(u_i) + \frac{1}{p}\vert u_i\vert^p \text{ and } h_J(u) = \frac{1}{2}\Vert A\Xi u - y \Vert^2_Y.
\end{equation*}
We are going to prove that $f_J=g_J+h_J$ is $p'$-conditioned on $\B_{\R^m}(\bar u,\delta)$, for any $\delta >0$.
To do so, we will apply to $f_J$ the sum rule in Theorem~\ref{T:sum rule}, which requires two hypotheses.
We must verify that the functions $g_J$ and $h_J$ are  conditioned up to linear perturbations (see equation \eqref{e:tilt}), and that the  qualification condition in \eqref{e:qcond} holds,
namely (since $\R^m$ is finite dimensional the strong relative interior coincides with the relative interior):
\begin{equation}\label{pli4}
0 \in \ri \left( \partial g_J^*(-\nabla h_J(\bar u)) - \partial h_J^*( \nabla h_J(\bar u)) \right).
\end{equation}
According to \eqref{pli2}, $\partial h_J^*( \nabla h_J(\bar u)) = \bar u + \ker S_J$.
Also, according to \cite[Proposition 13.30 \& Example 13.27(iii)]{BauComV2}, we have, for every $v\in\R^m$, $g_J^*(v) = \sum_{i=1}^m \frac{1}{q} \dist(v_i, I_{k_i})^q$, with $q=p/(p-1)$.
Since $t \mapsto \vert t \vert^q$ is continuously differentiable on $\R$, \cite[Example 17.33 and Proposition 17.31(ii)]{BauComV2}) imply that $g_J^*$ is G\^ateaux differetiable on $\R^m$.
This, together with the fact that $\ri \ker S_J = \ker S_J$, means that \eqref{pli4} is equivalent to
\begin{equation}\label{pli5}
0 \in \ri \left( \nabla g_J^*(-\nabla h_J(\bar u)) - \bar u + \ker S_J \right) = \nabla g_J^*(-\nabla h_J(\bar u)) - \bar u + \ker S_J.
\end{equation}
The latter inclusion holds true, since $\bar u \in \argmin f_J$ is equivalent to $\bar u \in \nabla g_J^*(-\nabla h_J(\bar u))$.
Thus, it only remains to prove that $\bar h_J = h_J - \langle \ell, \cdot \rangle $ and $\bar g_J = g_J - \langle \ell, \cdot \rangle$ are respectively $2$  and 
$p'$-conditioned on $\B_{\R^m}(\bar u,\delta)$, for $\ell$ being respectively in $R(\nabla h_J)$ and $R(\partial g_J)$.

Let us start with $\bar h_J$.
According to \eqref{pli2}, $\bar h_J$ is a  positive quadratic function being bounded from below, so it is $2$-conditioned on $\R^m$, with $\gamma_{\bar h_J,\R^m}$ 
being the smallest nonzero eigenvalue of $S_J$. Next, $\ell\in  R(\partial g_J)$ implies that there exists ${v}\in X$ such that $\ell\in\partial \bar{g}_J({v})$. Then, 
$0\in \partial \bar{g}_J((v))$, and this implies that ${v}$ is a minimizer of $\bar{g}_J$. It is also unique since $g_j$ is strictly convex.
If $v \notin \B_{\R^m}(\bar u,\delta)$, then $\bar g_J$ is automatically $p'$-conditioned on $\B_{\R^m}(\bar u,\delta)$, see for instance \cite[Proposition 3.3]{GarRosVil17}.
Assume then that $v \in \B_{\R^m}(\bar u,\delta)$, and use \cite[Proposition A.9]{ComSalVil15} to obtain the existence of $\gamma \in\left]0,+\infty\right[$ such that
\begin{equation*}
(\forall u \in \B_{\R^m}(\bar u,\delta))(\forall i\in\{1,...,m\}) \quad 
\frac{\gamma}{p'}\vert u_i - v_i \vert^{p'} \leq \frac{1}{p} \vert u_i \vert^p - \frac{1}{p} \vert v_i \vert^p - (u_i - v_i)\sgn (v_i) \vert v_i \vert^{p-1}.
\end{equation*}
Summing the above inequality over $i$, and using the fact that $\Vert \cdot \Vert_2^{p'} \leq \max\{1,m^{(p-2)/2}\} \Vert \cdot \Vert_{p'}^{p'}$, we derive by taking 
$\gamma'= \gamma \max\{1,m^{(p-2)/2}\}^{-1}$
that for all $u \in \B_{\R^m}(\bar u,\delta)$:
\begin{equation}\label{pli6}
\frac{\gamma'}{p'}\dist(u,\argmin \bar g_J)^{p'} 
\leq \sum\limits_{i=1}^{m} \frac{1}{p} \vert u_i \vert^p - \frac{1}{p} \vert v_i \vert^p - (u_i - v_i)\sgn (v_i) \vert v_i \vert^{p-1}.
\end{equation}
Introduce the following constant: $\omega_i=\sup I_{k_i}$ if $\ell_i > \sup I_{k_i}$, $\omega_i =\vert \ell_i \vert$ if $\ell_i \in I_{k_i}$, and 
$w_i=-\inf I_{k_i}$ if $\ell_i < \inf I_{k_i}$.
By making use of the first order condition at $v = \argmin \bar g_J$, it can be verified that 
\begin{equation*}\label{pli8}
(\forall _i \in \{1,...,m\}) \quad \vert v_i \vert^{p-1} = \vert \ell_i \vert - \omega_i, \,
 \sgn(v_i) = \sgn(\ell_i) 
  \text{ and }
 \sigma_{I_{k_i}}(v_i) = \omega_i \vert v_i \vert.
\end{equation*}
So we can deduce that
\begin{eqnarray*}
&  &
\frac{1}{p} \vert u_i \vert^p - \frac{1}{p} \vert v_i \vert^p 
- (u_i - v_i)\sgn (v_i) \vert v_i \vert^{p-1} \\
& = & 
\frac{1}{p} \vert u_i \vert^p - \frac{1}{p} \vert v_i \vert^p
-(u_i-v_i)(\ell_i - \sgn(v_i)\omega_i) \\
&=& \frac{1}{p} \vert u_i \vert^p - \frac{1}{p} \vert v_i \vert^p
-u_i \ell_i + v_i \ell_i - \sigma_{I_{k_i}}(v_i) + u_i \sgn(\ell_i) \omega_i.
\end{eqnarray*}
This, combined with \eqref{pli6}, leads to
\begin{equation*}
\frac{\gamma'}{p'}\dist(u,\argmin \bar g_J)^{p'} \leq 
\bar g_J(u) - \inf \bar g_J + \sum\limits_{i=1}^{m} -\sigma_{I_{k_i}}(u_i) +  u_i \sgn(\ell_i) \omega_i \leq \bar g_J(u) - \inf \bar g_J,
\end{equation*}
where the last inequality comes from the fact that $\sgn(\ell_i) \omega_i \in I_{k_i}$.
So we proved that $\bar g_J$ is $p'$-conditioned on $\B_{\R^m}(\bar u,\delta)$.
Theorem~\ref{T:sum rule} then yields that $f_J$ is $p'$-conditioned on $\B_{\R^m}(\bar u,\delta)$.
We conclude the proof applying Lemma \ref{L:finite to infinite conditioning}. 
\end{proof}

Combining Theorem \ref{T:CV thresholding gradient method} and Proposition \ref{P:p-LASSO is FUG}, be obtain rates for the corresponding thresholding gradient method.

\begin{theorem}\label{T:CV linear for pISTA}
Under the assumptions of Proposition~\ref{P:p-LASSO is FUG},  let $(x^n)_\nin$ be  the sequence generated by the forward-backward algorithm applied 
to $ f$.
Then $(x^n)_\nin$ converges strongly to some $\bar x \in \argmin f$.
If $p\in ]1,2[$, there exists two constants $\eps \in \left]0,1\right[$ and $C \in\left]0,+\infty\right[$, depending on $(\lambda,L,x^0,\I,A,y,p)$, such that
\begin{equation*}
(\forall n \geq 1) \quad f(x^n) - \inf f \leq \eps^n (f(x^0)-\inf f) \quad \text{ and } \quad  \Vert x^{n+1} - \bar x \Vert \leq C \sqrt{\eps}^n.
\end{equation*}
If $p \in ]2,+\infty[$, there exists two constants $(C_1,C_2)\in ]0,+\infty[^2$, depending on $(\lambda,L,x^0,\I,A,y,p)$, such that
\begin{equation*}
(\forall n \geq 1) \quad f(x^n) - \inf f \leq C_1{n^{-\frac{p}{p-2}}} \quad \text{ and }\quad  \Vert x_{n+1} - x_\infty \Vert \leq  C_2{n^{-\frac{1}{p-2}}}.
\end{equation*}
\end{theorem}

\section{Conclusion and perspectives}

In this paper we study and highlight the importance of the notion of extended support for minimization problems  with sparsity  inducing separable penalties. An identification result, together  with  uniform conditioning on finite dimensional sets, allow us to generalize and revisit 
classic convergence results for thresholding gradient methods, from a  novel and  different perspective, while further  providing new convergence rates. 

An interesting direction for future research would be to  go beyond separable penalties, in particular extending our results to regularizers promoting structured sparsity \cite{MosRosSan10}, such as group lasso.
A reasonable approach would be to extend the primal-dual arguments in \cite{FadMalPey17} to the infinite-dimensional setting.
A more challenging research direction seems the extension of our results   to gridless problems \cite{DuvPey17}.
Indeed, our analysis relies on the fact that the variables (signals) we consider are supported on a grid (indexed by $\NN \subset \N$), which allows  to use finite-dimensional arguments. 
Such an extension would require to work on Banach spaces of functions or of measures, and  seems an interesting venue for future research.

\section*{Acknowledgements}

\noindent This material is supported by the Center for Brains, Minds and Machines, funded by NSF STC award CCF-1231216, and the Air Force project FA9550-17-1-0390. 
G. Garrigos is supported by the European Research Council (ERC project NORIA), and part of his work was done while being a postdoc at the LCSL-IIT@MIT.
L. Rosasco acknowledges the financial support of the Italian Ministry of Education, University and Research FIRB project RBFR12M3AC. 
S. Villa is supported by the INDAM GNAMPA research project 2017 Algoritmi di ottimizzazione ed equazioni di evoluzione ereditarie.
L. Rosasco and S. Villa acknowledge the financial support from the EU project 777826 - NoMADS.

\appendix

\section{Annex}

\subsection{Closure, interior, boundary}

\begin{proposition}\label{P:derivative of support function}
Let $C\subset X$ be a closed convex set. Then:
\begin{enumerate}[(i)]
	\item\label{P:derivative of support function:0 gives full set} $\partial \sigma_C(0) = C$.
	\item\label{P:derivative of support function:exposed faces in boundary} For all $d \in X \setminus \{0\}$, \, $\partial \sigma_C(d) \subset \bd C$.
\end{enumerate}
Assume moreover that $\inte C \neq \emptyset$.
\begin{enumerate}[resume*]
	\item\label{P:relative interiors:nonempty interior case}  $\inte C =\qri C$.
\end{enumerate}
\end{proposition}
\begin{proof} \ref{P:derivative of support function:0 gives full set}: see \cite[Example 16.34]{BauComV2}.

\ref{P:derivative of support function:exposed faces in boundary}: see \cite[Proposition 7.3 \& Theorem 7.4]{BauComV2}.

\ref{P:relative interiors:nonempty interior case}: see \cite[Fact 6.14]{BauComV2}.
\end{proof}

\begin{proposition}\label{P:interior of infinite products}
Let $\I=(I_k)_\kinn$ be a collection of closed  proper intervals of $\R$, and suppose that $[-\omega,\omega] \subset I_k$ for all $\kinn$.
For every $x\in X$, let
\begin{equation}
\label{e:rho}
\rho(x)= \inf\limits_{x_k \in \inte I_k} \dist(x_k,\bd I_k).
\end{equation}
Then the following hold
\begin{enumerate}[(i)]
\item
\label{P:interior of infinite productsi}
For every $x \in X, \quad \rho(x) \in\left]0,+\infty\right[$;
\item 
\label{P:interior of infinite productsii}
$\inte \left(\bigoplus_{k\in\mathcal{N}} {I}_k\right)  = \bigoplus_{\kinn}\inte I_k$;
\end{enumerate}
\end{proposition}

\begin{proof} \ref{P:interior of infinite productsi}
Let $x\in X$. If $\mathcal{N}$ is finite the statement follows immediately. If $\mathcal{N}$ is infinite, since $\vert x_k \vert$ tends to $0$ when 
$k \to + \infty$, there exists  $K \in \mathcal{N}$ such that for all $k \geq K$, $\vert u_k \vert \leq \omega /2$.
Now, consider the following subsets of $\mathcal{N}$
\begin{equation*}
J=\{ \kinn \ | \ u_k \in \inte I_k \}, \ J_F= J \cap \{0,\dots , K-1\}, \ J_\infty= J \setminus J_F,
\end{equation*}
which are defined in such a way that $ \rho(x)=\inf_{k \in J} \dist(u_k , \bd I_k)$ and $J = J_F \sqcup J_\infty$.
Observe that $\rho(x) \leq \dist(x_K,\bd I_K) < + \infty$ since $\bd I_K\neq \varnothing$, so we only need to show that $\rho(x) > 0$.
Since $J_F$ is finite and $x_k \in \inte I_k$ for all $k \in J_F$, we have $\dist(x_k,\bd I_k) >0$ for all $k \in J_F$.
So we deduce that $\inf_{k \in J_F} \dist(x_k , \bd I_k) >0$.
On the other hand, for any $k \in J_\infty$, we have $\vert u_k \vert \leq \omega/2$, while $[-\omega,\omega] \subset I_k$, therefore $\dist(x_k,\bd I_k) \geq \omega/2$, 
and $\rho(x)=\inf_{k \in J} \dist(x_k , \bd I_k)  >0$ .

\ref{P:interior of infinite productsii}: let $x \in \inte \bigoplus_{k\in\mathcal{N}} {I}_k$. We are going to show that $x_k \in \inte I_k$ for all $k\in\mathcal{N}$.
By assumption, there exists $\delta\in\left]0,+\infty\right[$ such that $\B_X(x,\delta ) \subset \bigoplus_{k\in\mathcal{N}} {I}_k $.
Let $k \in \mathcal{N}$, and let us show that $[x_k-\delta,x_k+\delta] \subset I_k$.
Let  $y_k \in [x_k-\delta,x_k+\delta]$, and define $\bar x \in X$ such that $\bar x_k = y_k$ and $\bar{x}_i=x_i$ for every $i\neq k$.
Then we derive $\Vert x - \bar x \Vert = |x_k - y_k | = \delta$, whence $\bar x \in \B(x,\delta)\subset \bigoplus_{k\in\mathcal{N} }{I}_k $.
This implies that $y_k \in I_k$, which proves that $x_k \in \inte I_k$.
Now, we let $x \in \bigoplus_\kinn \inte I_k $, and we show that $x \in \inte\left( \bigoplus_{k\in\mathcal{N}} {I}_k\right)$.
By \ref{P:interior of infinite productsi}, $\rho(x)>0$ and, for every $\kinn$, $x_k\in\inte I_k$ by assumption. Let $\eta\in\left]0,\rho\right[$. Since
$\dist(x_k,\bd I_k)\geq \rho(x)$, we derive $[x_k-\eta,x_k+\eta] \subset I_k$.
On the other hand, the non-expansiveness of the projection implies that $\Vert x_k - p_k \Vert \leq \Vert x_k - y_k \Vert \leq \eta< \rho$, which leads to a contradiction.
Therefore $\B_X(x,\eta) \subset  \bigoplus_{\kinn} [x_k-\eta,x_k+\eta] \subset \bigoplus_{\kinn} {I}_k$. This yields $x \in \inte  \bigoplus_{\kinn} {I}_k$.
\end{proof}

\begin{proposition}[Quasi relative interior of infinite products]\label{P:qrinterior of infinite products}
Let $\I=(I_k)_\kin$ be a collection of closed  intervals of $\R$. Let $J\subset \mathcal{N}$ be a finite set, and suppose that $[-\omega,\omega] \subset I_k$ for all $\kinn\setminus J$.
Then
\begin{equation*}
\qri \bigoplus_{k\in\mathcal{N}} {I}_k  = \bigoplus_{\kinn} \ri I_k.
\end{equation*}
\end{proposition}

\begin{proof}
Assume $\mathcal{N}$ is infinite and set $J_\infty=\mathcal{N}\setminus J$.
We can then write
\begin{align*}
\qri  \bigoplus_{k\in\mathcal{N}} {I}_k& = &&\qri \bigg( (\bigoplus_{k \in J_\infty} {I}_k)\oplus(\bigoplus_{k \in J} {I}_k)\bigg)&& \text{ because $J$ is finite}, \\
	&= && \qri (\bigoplus_{k \in J_\infty} {I}_k)\oplus(\bigoplus_{k \in J} \qri I_k) && \text{ by  \cite[Proposition 2.5]{BorLew92}}, \\
	&= & &(\bigoplus_{k\in J_{\infty}} \inte I_k ) \oplus (\bigoplus_{k \in J} \qri I_k) && \text{ by  Proposition \ref{P:interior of infinite products}}, \\
	&=& &  \bigoplus_{\kinn} \ri I_k && \text{ by  Proposition \ref{P:derivative of support function}\ref{P:relative interiors:nonempty interior case}}.
\end{align*}
\end{proof}

\subsection{Functions}

\begin{lemma}\label{L:prox is injective at smooth minimizer}
Let $\psi \in \Gamma_0(X)$ be differentiable at $0 \in \argmin \psi$ and let $x\in X$.
Then
$$x = 0 \Leftrightarrow \prox_\psi(x) =0.$$
\end{lemma}

\begin{proof}
$\prox_\psi(x) = 0 \Leftrightarrow (Id + \partial \psi)^{-1}(x)=0 \Leftrightarrow x \in 0 + \partial \psi (0) \Leftrightarrow x = \nabla \psi(0) \Leftrightarrow x=0$.
\end{proof}

\begin{proposition}\label{P:countable sum of functions}
Let $g_k \in \Gamma_0(\R)$ with $\inf g_k=g_k(0)=0$ for all $\kinn$. 
Define $g\colon X\to \mathbb{R}\cup\{+\infty\}\colon x\mapsto\sum_\kinn g_k(x_k)$.  Then:
\begin{enumerate}[(i)]
	\item\label{P:countable sum of functions:convex} $g\in \Gamma_0(\X)$.
	\item\label{P:countable sum of functions:domain ssdiff} $\dom \partial g = \{ x \in \X \ | \ \bigoplus_{\kinn} \partial g_k (x_k) \neq \emptyset \}$.
	\item\label{P:countable sum of functions:ssdiff} For all $x \in \dom \partial g$, $\partial g(x) = \bigoplus_{\kinn} \partial g_k (x_k)$.
	\item\label{P:countable sum of functions:prox} For all $x \in \X$,  $\prox_g(x) =\sum_{\kinn} \prox_{g_k}(x_k) e_k$.
\end{enumerate}
\end{proposition}

\begin{proof}
\ref{P:countable sum of functions:convex}: $g$ is convex by definition. It is proper because $g(0)=0$ and $g\geq 0$. Fatou's lemma implies that $g$  is lower semicontinuous.

\ref{P:countable sum of functions:domain ssdiff}-\ref{P:countable sum of functions:ssdiff}: follow directly from the fact that
\begin{eqnarray*}
(\forall (x^*,x) \in X^2) \quad x^* \in \partial g (x) & \Leftrightarrow & (\forall y \in X) \quad g(y) - g(x) - \langle x^* , y - x \rangle \geq 0 \\
	 & \Leftrightarrow & (\forall y \in X) \quad \sum\limits_{\kin}^{} g_k(y_k) - g_k(x_k) - \langle x_k^* , y_k - x_k \rangle \geq 0 \\
	 	 & \Leftrightarrow & (\forall \kinn) \quad x_k^* \in \partial g_k(x_k),
\end{eqnarray*}
where the last equivalence holds by taking for all $\kinn$ \, $y_i=x_i$ if $i\neq k$.\

\ref{P:countable sum of functions:prox}: let $(x,p) \in X^2$. It follows from \ref{P:countable sum of functions:ssdiff} that
\begin{align*}
p=\prox_g(x) &\iff p-x\in\partial g(p)\\
&\implies(\forall \kinn)\quad p_k-x_k \in\partial g_k(p_k)\\
&\iff (\forall \kinn)\quad p_k=\prox_{g_k}(x_k). 
\end{align*}
\end{proof}

\begin{proposition}
\label{P:countable sum of support functions}
Let $\I=(I_k)_\kin$ is a family of proper closed interval of $\R$.
Let, for every $x\in X$, $g(x) =\sum_\kinn \sigma_{I_k}(x_k)$. Then the 
following hold.
\begin{enumerate}[(i)]
	\item\label{P:countable sum of support functions:coercivity equivalence} $g$ is coercive if and only if $0 \in \inte I_k$ for all $\kinn$.
\end{enumerate}
Assume moreover that there exists $\omega > 0$ such that $[-\omega, \omega] \subset I_k$ for all $\kinn$. Then
\begin{enumerate}[resume*]
	\item\label{P:countable sum of support functions:coercivity and support box} $g \in \Gamma_0(X)$ is coercive and $g$ is the support function of ${\B}_{\infty,\I}=\bigoplus_{\kinn} {I}_k$,
	\item\label{P:countable sum of support functions:dom ssdif} $\dom \partial g = c_{00}$ and $\dom \partial g^* = {\B}_{\infty,\I}$,
	\item\label{P:countable sum of support functions:prox} for every $x \in X$, and for every $\lambda >0$, \, $\prox_{\lambda g}(x) = \left( x_k - \lambda \proj_{I_k}(\lambda^{-1} x_k) \right)_\kinn$.
\end{enumerate}	 
\end{proposition}

\begin{proof} 
\ref{P:countable sum of support functions:coercivity equivalence}: observe that
\begin{eqnarray*}
\text{$g$ is coercive} & \Leftrightarrow & (\forall \kinn) \quad \text{$\sigma_{I_k}$ is coercive} \quad \text{(take $x_k=0$ except for one index $k$)} \\
	& \Leftrightarrow & (\forall \kinn) \quad 0 \in \inte \dom \sigma_{I_k}^* \quad \text{by \cite[Proposition 14.16]{BauComV2}} \\
	& \Leftrightarrow & (\forall \kinn) \quad 0 \in \inte I_k \quad \text{since $\sigma_{I_k}^* = \delta_{I_k}$.}
\end{eqnarray*}
\ref{P:countable sum of support functions:coercivity and support box}: assume that $\mathcal{N}$ is infinite.
Item \ref{P:countable sum of support functions:coercivity equivalence} implies that $g \in \Gamma_0(X)$ and is coercive.
To prove that $g$ is the support function of ${\B}_{\infty,\I}$, we will show that $g^*$ is its indicator function.
Let $x^* \in X$. Then
\begin{align*}
g^*(x^*) &= \sup\limits_{x \in X } \langle x^* , x \rangle - g(x) =\sup\limits_{x \in X}  \sum\limits_{\kinn}^{}\langle x^*_k, x_k \rangle - \sigma_{I_k}(x_k) \\
&\leq  \sum\limits_{\kinn}^{} \sup\limits_{x_k \in X_k} \langle x^*_k, x_k \rangle - \sigma_{I_k}(x_k)= \sum\limits_{\kinn}^{} \sigma_{I_k}^*(x_k^*) = \delta_{{\B}_{\infty,\I}}(x^*)
\end{align*}
To prove the converse inequality, since $x^* \in X$, there exists some $K \in\mathcal{N}$ such that for all $k \geq K$, $\Vert x^*_k \Vert < \omega$, meaning that $x^*_k \in I_k$, 
and therefore $\delta_{I_k}(x_k^*)=0$.
Let $J_K=\{0,\dots , K-1\}$. Since we deal with a finite sum, 
\begin{eqnarray*}
\delta_{{\B}_{\infty,\I}}(x^*) & =& \sum\limits_{k \in J_K}^{} \delta_{I_k}(x_k^*) 
= \sum\limits_{k \in J_K}^{} \sup\limits_{x_k \in \R} \langle x_k^*,x_k\rangle - \sigma_{I_k}(x_k) 
= \sup\limits_{x \in X_{J_K}} \sum\limits_{k \in J_K}^{} \langle x_k^*,x_k\rangle - \sigma_{I_k}(x_k) .
\end{eqnarray*}
Moreover, setting $J_\infty=\mathcal{N} \setminus {J_K}$: $$\sup\limits_{x \in X_{J_\infty}} \sum\limits_{k \in J_\infty}^{} \langle x_k^*,x_k\rangle - \sigma_{I_k}(x_k) \geq 0,$$
and this yields
\begin{equation*}
\delta_{{\B}_{\infty,\I}}(x^*)  \leq \sup\limits_{x \in X_{J_K}} \sum\limits_{k \in J_K}^{} \langle x_k^*,x_k\rangle - \sigma_{I_k}(x_k) + \sup\limits_{x \in X_{J_\infty}} \sum\limits_{k \in J_\infty}^{} \langle x_k^*,x_k\rangle - \sigma_{I_k}(x_k) = g^*(x^*).
\end{equation*}
\ref{P:countable sum of support functions:dom ssdif}:  assume that $\mathcal{N}$ is infinite. The equality $\dom \partial g^* = {\B}_{\infty,\I}$ follows from \ref{P:countable sum of support functions:coercivity and support box}.
It remains to show that $\dom \partial g = c_{00}$.
Let $x \in \dom \partial g$. By Proposition \ref{P:countable sum of functions}\ref{P:countable sum of functions:domain ssdiff} there exists $x^* \in \bigoplus_{\kinn} \partial \sigma_{I_k}(x_k)$.
For all $\kinn$, Proposition \ref{P:derivative of support function}\ref{P:derivative of support function:0 gives full set}-\ref{P:derivative of support function:exposed faces in boundary} 
yields that $\partial \sigma_{I_k}(x_k) = I_k$ if $x_k = 0$, and $\partial \sigma_{I_k}(x_k) \subset \bd I_k$ if $x_k \neq 0$.
Assume by contradiction that $x \notin c_{00}$, i.e. there exists $k_n \to + \infty$ such that $x_{k_n} \neq 0$ for all $n\in\N$.
Then, it follows that $x_{k_n}^* \in \bd I_{k_n}$ for all $\nin$, and therefore $\Vert x_{k_n}^* \Vert \geq \omega$, which  contradicts the fact that $x^* \in X$.
Now, let $x \in c_{00}$ and let $K \in \mathcal{N}$ be such that $x_k =0$ for all $k \geq  K$ and let $\mathcal{N}_K=\mathcal{N}\cap\{0,\ldots,K\}$. 
By Proposition \ref{P:derivative of support function}\ref{P:derivative of support function:0 gives full set}
$\partial \sigma_{I_k}(x_k) = I_k \ni 0$ for all $k \geq  K$, therefore
\begin{equation*}
\emptyset \neq \bigoplus_{k\in\mathcal{N}_K}\partial \sigma_{I_k}(x_k) \subset \partial g(x).
\end{equation*}
\ref{P:countable sum of support functions:prox}: is a direct consequence of Proposition \ref{P:countable sum of functions}\ref{P:countable sum of functions:prox} and Moreau's identity \cite[Theorem 14.3.ii]{BauComV2}.
\end{proof}

We recall a sum rule for conditioning,  obtained in \cite[Theorem 3.1]{GarRosVil17}.
\begin{theorem}\label{T:sum rule}
Let $f=g+h$, where $g \in \Gamma_0(X)$ and $h \in \Gamma_0(X)$ is of class $C^1$.
Let $\Omega\subset X$.
Assume that there exists $\bar x \in \argmin f$ such that, for $\bar v=-\nabla h (\bar x)$ and $(p_1,p_2)\in \left[1,+\infty\right[^2$,
\begin{equation}
\label{e:tilt}
\text{ $g - \langle \bar v, \cdot \rangle$ is $p_1$-conditioned on $\Omega$ and $h + \langle \bar v, \cdot \rangle$ is $p_2$-conditioned on $\Omega$. }
\end{equation}
Suppose that 
\begin{eqnarray}
\label{e:qcond}
	0 & \in & \sri(\partial g^*(\bar v) - \partial h^* ( - \bar v) ), \label{D:QC2}
\end{eqnarray}
and let $p=\max\{p_1,p_2\}$. Then, for any $\delta \in \left]0, + \infty\right[$, $f$ is $p$-conditioned on $\Omega \cap  \B_X(0,\delta)$.
\end{theorem}

\subsection{Auxiliary results}
\begin{lemma}
Let $\{x^1,...,x^N\} \subset X$ be a finite family.
Then there exists $\bar x \in \co \{ x^1,...,x^N \}$ such that $\supp(\bar x) = \cup \{ \supp(x^i) \ | \ i\in \{1,...,N\} \}$.
\end{lemma}

\begin{proof}
We proceed by induction.
If $N=1$ this is trivially true.
Let us turn on the $N=2$ case, by considering $\{x^1,x^2\}$ in $X$.
If $\supp(x^1) = \supp(x^2)$, then it is enough to take $\bar x = x^1$ or $\bar x = x^2$.
Assume that $\supp(x^1) \neq \supp(x^2)$.
Define
\begin{equation*}
\Lambda = \left\{ \frac{\vert x_k^2 \vert }{\vert x^2_k-x_k^1 \vert} \ | \ \kinn, \,  x^1_k  x^2_k<0  \right\}.
\end{equation*}
$\Lambda$ is well defined because $ x^1_kx^2_k <0$ implies that $x^2_k - x^1_k \neq 0$.
Moreover, $\Lambda \subset \left]0,1\right[$, and is at most countable.
Let $\lambda \in \left]0,1\right[ \setminus \Lambda$, and define $\bar x = \lambda x^1 + (1- \lambda) x^2$.
By definition we have $\bar x \in \co \{ x^1,x^2\}$, so it remains to check that $\supp(\bar x) = \supp(x^1) \cup \supp(x^2)$. To prove this, first assume that $k \in \supp(\bar x)$.
If $k\in \supp(x^1)$ it is trivial, so assume that $k\notin \supp(x^1)$.
In that case $\bar x_k = \lambda \cdot 0 + (1-\lambda) x^2_k$, where $\lambda \neq 1$ and $\bar x_k \neq 0$, from which we deduce that $k \in \supp(x^2)$.
This shows that $\supp(\bar x) \subset\supp(x^1) \cup \supp(x^2)$
Now, take $k \in \supp(x^1) \cup \supp(x^2)$, and assume by contradiction that $\bar x_k =0$.
Then 
\begin{equation*}
x_k^1\neq 0, x_k^2 \neq 0, x_k^1 = (1 - \lambda^{-1}) x_k^2, \text{ and } \lambda = \frac{\vert x_k^2 \vert }{\vert x^2_k-x_k^1 \vert},
\end{equation*}
which contradicts the fact that $\lambda \notin \Lambda$. Therefore $\supp(x^1) \cup \supp(x^2)\subset \supp(\bar x) $.
Assume now that the statement  holds for $N \geq 2$, and let us prove it for $N+1$.
Let $\{x^1,...,x^N, x^{N+1}\} \subset X$ be a finite family.
By inductive hypotheses we can find some $\bar x^1 \in \co \{x^1,...,x^N\}$ such that $\supp(\bar x^1) = \cup \{ \supp(x^i) \ | \ i\in \{1,...,N\} \}$.
Moreover, the inductive hypotheses guarantees the existence of some $\bar x \in \co \{ \bar x^1,x^{N+1}\}$ such that $\supp(\bar x) = \supp(\bar x^1)  \cup \supp(x^{N+1})$.
We derive from the definition of $\bar x^1$ that $\supp(\bar x) =\cup \{ \supp(x^i) \ | \ i\in \{1,...,N+1\} \}$.
Also, $\bar x^1 \in \co \{x^1,...,x^N\}$ and $\bar x \in \co \{ \bar x^1,x^{N+1}\}$  imply that $\bar x \in\co \{x^1,...,x^N,x^{N+1}\}$, which ends the proof.
\end{proof}

\begin{lemma}\label{L:existence point support convex hull}
Let $C \subset X$ be a convex nonempty set, and $J= \cup \{ \supp(x) \ | \ x \in C \}$.
If $J$ is finite, then there exists $\bar x \in C$ such that $\supp (\bar x) = J$.
\end{lemma}

\begin{proof}
Since $J$ is finite, there exists a finite family $\{x^1,...,x^N\} \subset C$ such that $J = \cup \{ \supp(x^i) \ | \ i\in \{1,...,N\} \}$.
It suffices then to apply the previous lemma to obtain such $\bar x \in  \co \{ x^1,...,x^N \} \subset C$.
\end{proof}

\end{document}